\documentclass[11pt, leqno]{article}
\usepackage{amsmath,amsfonts,amssymb,wasysym}%,hyperref}
\usepackage[latin1]{inputenc}
\usepackage{shapepar}
\usepackage{graphicx}
\usepackage{hyperref}
\usepackage{ifpdf} %for .eps and .jpg
\usepackage{color}
\newcommand{\R}{{\mathbb R}}
\newcommand{\re}{{\mathbb R}}
\newcommand{\ren}{{\mathbb R}^N}

\newcommand{\be}[1]{\begin{equation}\label{#1}}
\newcommand{\ee}{\end{equation}}

\newcommand{\prf}{\par\smallskip\noindent{\sl Proof. \/}}
\newcommand{\finprf}{\unskip\null\hfill$\;\square$\vskip 0.3cm}

\newenvironment{proof}{\prf}{\finprf}
\newtheorem{theorem}{Theorem}[section]
\newtheorem{lemma}{Lemma}[section]

\newtheorem{proposition}[theorem]{Proposition}

\newtheorem{definition}{Definition}[section]

\newcommand{\ve}{\varepsilon}
\numberwithin{equation}{section}

 \newcommand{\nc}{\normalcolor}
\def\qed{\,\unskip\kern 6pt \penalty 500
\raise -2pt\hbox{\vrule \vbox to8pt{\hrule width 6pt
\vfill\hrule}\vrule}\par}
\definecolor{darkblue}{rgb}{0.05, .05, .65}
\definecolor{darkgreen}{rgb}{0.1, .65, .1}
\definecolor{darkred}{rgb}{0.8,0,0}
\begin{document}
\title{\textbf{ Symmetrization for Linear and Nonlinear \\ Fractional  Parabolic Equations \\ of Porous Medium Type }\\[7mm]}

\author{\Large  Juan Luis V\'azquez \footnote{Departamento de Matem\'aticas, Universidad Aut\'onoma de Madrid, 28049 Madrid, Spain. \newline
E-mail:
{\tt juanluis.vazquez@uam.es}} \ and \ Bruno Volzone \footnote{Dipartimento per le Tecnologie, Facolt\`a di Ingegneria, Università degli Studi di
Napoli "Parthenope", 80143  Italia. \    E-mail: {\tt bruno.volzone@uniparthenope.it}} }

\date{} %%  this cancels date in article format

\maketitle

\begin{abstract}
We establish symmetrization results for the solutions of the linear  fractional diffusion equation $\partial_t u +(-\Delta)^{\sigma/2}u=f$ and its
elliptic counterpart
$h v +(-\Delta)^{\sigma/2}v=f$, $h>0$, using the concept of comparison of concentrations. The results extend to the nonlinear version, $\partial_t u
+(-\Delta)^{\sigma/2}A(u)=f$,
but only when $A:\re_+\to\re_+$ is a concave function. In the elliptic case, complete symmetrization results are  proved for $\,B(v)
+(-\Delta)^{\sigma/2}v=f$ \ when $B(v)$ is a convex nonnegative function for $v>0$ with $B(0)=0$, and partial results when $B$ is concave.
Remarkable counterexamples are constructed for the parabolic equation when $A$ is convex, resp. for the elliptic equation when $B$ is concave. Such counterexamples do not exist in the standard diffusion case $\sigma=2$.
\end{abstract}

\setcounter{page}{1}
%%%%%%%%%%%%%%%%%%%%%%%%%%%%%%%%%%%%%%%%%%%%%%%%%%%%%%%%%%%%%%%%%
\section{Introduction}\label{sec.intro}

 The techniques of symmetrization are a very popular tool of obtaining a priori estimates for the solutions of different partial differential  equations,  notably  those of elliptic and parabolic type. Symmetrization techniques appear in classical works like \cite{MR0046395, PS1951}. The application of  Schwarz  symmetrization to obtaining a priori estimates for elliptic problems is described by Weinberger in \cite{Wein62}, see also  \cite{Maz}.  The  standard  elliptic result refers to the solutions of an equation of the form
$$
Lu=f,  \qquad Lu=-\sum_{i,j} \partial_i(a_{ij}\partial_j u)\,,
$$
posed in a bounded domain $\Omega\subseteq \ren$;  the coefficients $\{a_{ij}\}$ are assumed to be bounded, measurable and satisfy the usual ellipticity condition; finally, we take zero Dirichlet boundary conditions on the boundary $\partial\Omega$. The now classical analysis, introduced by Talenti \cite{Talenti1}, leads to pointwise comparison between the symmetrized version of the actual solution of the problem $u(x)$ and
the radially symmetric solution $v(|x|)$ of some radially symmetric model problem which is posed in a ball with the same volume as $\Omega$. Sharp a priori estimates for the solutions are then derived. We refer to the papers \cite{Talenti4, ATL90}  in the framework of linear operators, and \cite{Talenti3, FerMess} where comparison results are obtained when dealing with nonlinear elliptic operators of divergence type.
There is a large literature in this topic with many interesting developments.

When this technique is applied to parabolic equations, the general program of comparison with a model problem of radial type still works, but the result of pointwise comparison need not hold and has to be replaced by comparison of $L^p$ norms at every time $t>0$. Actually, a more basic result, called
{\sl comparison of concentrations} is true, cf. Bandle \cite{Bandle, Band2} where linear parabolic equations with smooth coefficients are discussed.
Such results have been extended in  works like  \cite{Mossino}, \cite{AlvVolpVolz1}, \cite{VolpVolz} for weak solutions of linear parabolic problems with discontinuous coefficients.   Relevant definitions about symmetrization,  rearrangements and  concentration are recalled in Section~\ref{sec.prelim}.

\medskip

\noindent {\sc Elliptic approach to nonlinear parabolic problems.} An extension of the symmetrization results to nonlinear parabolic equations of possibly degenerate type, more precisely of the porous medium type, was done by the first author in \cite{Vsym82}. The paper considers the evolution
problem
\begin{equation}\label{evol.pbm}
\partial_t u=\Delta A(u), \quad u(0)=u_0,
 \end{equation}
where $A$ a monotone increasing real function\footnote{More generally, $A$ can  be a maximal monotone graph, but that generality  is of no concern for us here.}
and $u_0$ is a suitably given initial datum which is assumed to be integrable. For simplicity the problem was posed for $x\in \ren$,  but bounded open sets can be used as spatial domains. The
novel idea of the paper was to use the famous Crandall-Liggett Implicit Discretization theorem \cite{CL71} to reduce the evolution problem to a sequence of nonlinear elliptic problems of the iterative form
\begin{equation}
- h\,\Delta A(u(t_k))+u(t_k)=u(t_{k-1}),\quad k=1,2, \cdots,
\end{equation}
where $t_k=kh$, and $h>0$  is the time step per iteration. Writing $A(u)=v$, the resulting chain of elliptic problems can be written in the common form
\begin{equation}\label{ell.eq}
h \,L v + B(v)= f\,, \quad B=A^{-1}.
\end{equation}
General theory of these equations, cf.  \cite{BBC1975}, ensures that the solution map: $$T:f\mapsto u=B(v)$$ is a contraction in some Banach space, which happens to be $L^1(\Omega)$. Note that the constant $h>0$ is not essential, it can be put to 1 by scaling. In that context, the symmetrization
result can be split into two results:

(i) the first one applies to rearranged right-hand sides and solutions.  It says that
if two r.h.s. functions $f_1, f_2$, are rearranged and satisfy a concentration comparison of the form $f_1 \prec f_2$, then the same applies to the solutions, in the form $B(v_1)\prec B(v_2)$.\footnote{See the definitions of  the order relation $\prec$ in Section \ref{sec.prelim}.}

(ii) The second result  aims at comparing the solution $v$ of equation \eqref{ell.eq} with  a non-rearranged function $f$  with the solution $\tilde v$ corresponding to $f^\#$, the  radially decreasing rearrangement of $f$. We obtain that  $\tilde v$ is a rearranged function and $ B(v^\#)\prec
B(\tilde v)$, i.\,e., $B(v)$ is less concentrated than $B(\tilde v)$.

This precise pair of comparison results can be combined to obtain similar results along the whole chain of iterations $u(t_k)$ of the evolution process, if discretized as indicated above. This allows in turn to conclude the symmetrization theorems (concentration comparison and comparison of $L^p$ norms)
for the evolution problem \eqref{evol.pbm}. This approach has had a large expansion in the past decades, cf. \cite{VANS05} and references. There is no difficulty in considering equations with a right-hand side, like $u_t=\Delta A(u)+f$, as long as $f\in L^1(Q_T)$, $Q_T=\ren\times(0,T)$. We also mention how time
discretization and symmetrization tools can be combined together to get interesting comparison results for some types of parabolic equations with double nonlinearity, such as $b(u)_t=\Delta A(u)+f$, with special assumptions on $b$, see \cite{Diaz1} and  \cite{AlvVolpVolz2}, and to equations with weights \cite{Reyes}. The technique also applies for $p$-Laplacian operators, cf. \cite{Vport}.

\medskip

\noindent {\sc Equations with fractional operators.} The study of elliptic and parabolic equations involving nonlocal operators, usually of fractional type, is currently the subject of great attention. In that sense, it is quite natural to investigate how to apply symmetrization techniques to the
elliptic equations like
\begin{equation}
(-\Delta)^{\sigma/2}v=f,
\end{equation}
where the standard Laplacian operator $\Delta$ is replaced by one of the fractional Laplacian operators $(-\Delta)^{\sigma/2}$, $0<\sigma<2$, as defined in \cite{Landkof72, Stein70}.
The study of this question has been successfully implemented by the second author and Di Blasio in a recent paper \cite{BV} by an interesting technique of analysis of Steiner symmetrization of an equivalent extended problem, based on the  extension technique used by Caffarelli and Silvestre for the definition of
$\sigma$-Laplacian operator, \cite{CaffS}. For previous uses of Steiner symmetrization in standard elliptic problems see \cite{ATLD}. The results of \cite{BV} include a comparison of
concentrations, in the form \ $v^\#\prec \tilde v $, that parallels the result that holds in the standard Laplacian case; note however that no pointwise comparison is obtained, so the result looks a bit like the parabolic results of the standard theory as mentioned above.

In the present paper we are interested in considering the application of such symmetrization techniques
to linear or nonlinear parabolic equations with similar fractional Laplacian operators. To be specific, we will focus on the equations of the form
\begin{equation}\label{nolin.parab}
\partial_t u +(-\Delta)^{\sigma/2}A(u)=f, \qquad 0<\sigma<2\,.
\end{equation}
Following the theory for the standard Laplacian just sketched, we want to consider as nonlinearity $A$ an increasing real function such that $A(0)=0$, and we may accept some other regularity conditions as needed, like $A$ smooth with $A'(u)>0$ for all $u>0$. The problem is posed in the whole space $\ren$. We want to pay special attention to the form $A(u)=u^m$ with $m>0$; the equation is then
called the
Fractional
Heat Equation (FHE) when $m=1$, the Fractional Porous Medium Equation (FPME)  if $m>1$, and the Fractional Fast Diffusion Equation (FFDE) if $m<1$.
Let us
recall
that the linear equation \ $\partial_t u +(-\Delta)^{\sigma/2}u=0$  is a model of so-called anomalous diffusion, a much studied topic in physics, see
for instance \cite{AT, JKOlla, MMM, VIKH, WZ} and the references therein.  The interest in these operators has a long history in
Probability since the fractional Laplacian operators of the form $(-\Delta)^{\sigma/2}$  are infinitesimal generators of stable L\'{e}vy processes,
see \cite{Applebaum, Bertoin, Valdinoc}.

For $A(u)=u^m$ and general $m>0$ we obtain a nonlinear diffusion model; the theory of existence of weak solutions for the initial value problem has
been addressed by the first author and collaborators in \cite{pqrv, pqrv2, vazBaren}, and the main properties have been obtained. In particular, if
we
take
initial data in $L^1$, then an $L^1$-contraction semigroup is generated, and the Crandall-Liggett discretization theorem applies. Extension of this
results to general smooth $A$ is done in
\cite{pqrv4}.

The application of the method of implicit time discretization leads to the nonlinear equation of elliptic type
\begin{equation}\label{nolin.ell}
h\,(-\Delta)^{\sigma/2}v+B(v) =f
\end{equation}
posed again in the whole space $\ren$ or in an open subdomain $\Omega\subset \ren$ with zero Dirichlet boundary conditions; $h>0$ is a non-important
constant,
and the nonlinearity $B$ is the inverse function to the monotone function $A$ that appears in the parabolic equation.

 \medskip

\noindent {\sc Organization of the paper and main results.} Section \ref{sec.prelim} contains the preliminaries about symmetrization and mass
concentration that we will need.

\noindent $ \bullet$ As a first step of our analysis, we address in Section \ref{sect.ell} the issue of comparison of concentrations for rearranged
functions,
more precisely how to compare the rearrangement of the solution of an elliptic problem with data $f$ with the solution of  a radial problem with
data $f^{\#}$, rearrangement of $f$. The technique used in \cite{BV} does not work  for the
modified equation with lower order term \eqref{nolin.ell}. We supply in this paper the proof of elliptic concentration comparison in the two forms
that
are
needed to try pass to the parabolic problem via discretization in time. However, the results are complete only in the case where $B$ is a convex
function
and
$\Omega=\ren$. Though the elliptic results are used here as a step towards the parabolic theory, they have an interest in themselves as an
improvement
on the symmetrization result developed  in \cite{BV}.

\noindent $ \bullet$ Complementing this analysis, we  prove in Section \ref{sec.ell.count} that one the elliptic comparison results that is needed to
build a good parabolic theory is false in the case of a concave $B$ of the form $B(v)=v^m$, $0<m<1$.

\noindent $ \bullet$ The main issue of symmetrization for linear or nonlinear fractional parabolic equations is addressed in Section \ref{sec.par}. After the iteration steps described above, the elliptic results allow to conclude  similar comparison results for the mild solutions of the evolution problem \eqref{nolin.parab} when $A$ is concave, i.\,e., in the range of exponents $0<m<1$ of the Fractional Fast Diffusion Equation (FFDE), and also in the most popular case, the Fractional Heat Equation (FHE)
\begin{equation}\label{ell.eq2}
\partial_t u +(-\Delta)^{\sigma/2}u=f, \qquad u(0)=u_0.
\end{equation}

A number of consequences are derived from the symmetrization result in the form of a priori estimates,  much in the manner these consequences are derived in the case of equation involving the standard Laplacian. In particular, we can use the Barenblatt solutions of the FHE and FFDE constructed in
\cite{vazBaren} as a worst-case to obtain a priori estimates for the solutions in the $L^p$ spaces, a result that was one of the main corollaries of paper \cite{Vsym82}. Such consequences are important in developing the general theory of such equations, which is one of the aims of the symmetrization
techniques.  In order to keep a reasonable length for this paper, we have decided to explain such consequences in a companion paper, \cite{VazVol2}.

\noindent $ \bullet$ Returning to the presentation of the main results, an important gap was therefore left in the analysis, namely to examine what happens with this approach when applied to the FPME with $m>1$, or more generally to \eqref{nolin.parab} when $A$ is not concave. To our surprise, the result
of comparison of concentrations  is false for the evolution problem, i.\,e., for the FPME with $m>1$. As a consequence, we cannot use the Barenblatt solutions of the FPME as a worst case to obtain a priori estimates for the solutions in the $L^p$ spaces. The surprising negative results about mass comparison for parabolic equations are described in Section \ref{sec.neg.par}. This is to be seen in parallel to the counterexample found for the elliptic equation \eqref{ell.eq2}; by the way, this one was found later and is less intuitive.

We conclude by a short section containing comments, extensions and open problems.

%%%%%%%%%%%%%%%%%%%%%%%%%%%%%%%%%%%%%%%%%%%%%%%%%%%%%%%%%%%%%%%%%%%%%%%%%%%%%%%%%%%%%%%
%%%%%%%%%%%%%%%%%%%%%%%%%%%%%%
\section{Preliminaries on symmetrization}\label{sec.prelim}
\setcounter{equation}{0}

A measurable real function $f$ defined on $\R^{N}$ is called \emph{radially symmetric} (\emph{radial},  for short) if there is a function
$\widetilde{f}:[0,\infty)\rightarrow \R$ such that $f(x)=\widetilde{f}(|x|)$ for all $x\in \R^{N}$. We will often write $f(x)=f(r)$,
$r=|x|\ge0$ for
such functions by abuse of notation. We say that $f$ is \emph{rearranged} if it is radial, nonnegative and $\widetilde{f}$ is a
right-continuous,
non-increasing
function of $r>0$. A similar definition can be applied for real functions defined on a ball $B_{R}(0)=\left\{x\in\R^{N}:|x|<R\right\}$.

Now, let $\Omega$ be an open set of $%
%TCIMACRO{\U{211d} }%
%BeginExpansion
\mathbb{R}
%EndExpansion
^{N}$ and $f$ be a real measurable function on $\Omega$. We will denote by $\left\vert
\cdot\right\vert $ the $N$-dimensional Lebesgue measure. We define the
\emph{distribution function} $\mu_{f}$ of $f$ as%
\[
\mu_{f}\left(  k\right)  =\left\vert \left\{  x\in\Omega:\left\vert f\left(
x\right)  \right\vert >k\right\}  \right\vert \text{ , }k\geq0,
\]
and the \emph{decreasing rearrangement} of $f$ as%
\[
f^{\ast}\left(  s\right)  =\sup\left\{ k\geq0:\mu_{f}\left(  k\right)
>s\right\}  \text{ , }s\in\left(  0,\left\vert \Omega\right\vert \right).
\]
We may also  think of extending $f^{\ast}$ as the  zero function in $[|\Omega|,\infty)$ if $\Omega$ is bounded. From this definition it turns out
that $\mu_{f^{\ast}}=\mu_{f}$ (\emph{i.\,e.\,,} $f$, and $f^{\ast}$ are equi-distributed) and $f^{\ast}$ is exactly  the \emph{generalized
inverse} of
$\mu_{f}$.
Furthermore, if $\omega_{N\text{ }}$ is the measure of the unit ball in $%
%TCIMACRO{\U{211d} }%
%BeginExpansion
\mathbb{R}
%EndExpansion
^{N}$ and $\Omega^{\#}$ is the ball of $%
%TCIMACRO{\U{211d} }%
%BeginExpansion
\mathbb{R}
%EndExpansion
^{N}$ centered at the origin having the same Lebesgue measure as $\Omega,$ we define the
function
\[
f^{\#}\left(  x\right)  =f^{\ast}(\omega_{N}\left\vert x\right\vert
^{N})\text{ \ , }x\in\Omega^{\#},
\]
that will be called \emph{spherical decreasing rearrangement} of $f$. From this definition it follows that $f$ is rearranged if and only if
$f=f^{\#}$.

For an exhaustive treatment of rearrangements we refer to \cite{Bandle}, \cite{Kawohl}, or the appendix in \cite{Talenti2}. Here, we just recall
the conservation of the
$L^{p}$
norms (coming from the definition of rearrangements and the classical \emph{Cavalieri principle}): for all $p\in[1,\infty]$
\[
\|f\|_{L^{p}(\Omega)}=\|f^{\ast}\|_{L^{p}(|0,\Omega|)}=\|f^{\#}\|_{L^{p}(\Omega^{\#})}\,,
\]
as well as the classical Hardy-Littlewood inequality (see \cite{MR0046395})%
\begin{equation}
\int_{\Omega}\left\vert f\left(  x\right)  g\left(
x\right)  \right\vert dx\leq\int_{0}^{\left\vert \Omega\right\vert }f^{\ast
}\left(  s\right)  g^{\ast}\left(  s\right)  ds=\int_{\Omega^{\#}}f^{\#}(x)\,g^{\#}(x)\,dx\,,
\label{HardyLit}%
\end{equation}
where $f,g$ are measurable functions on $\Omega$.

\noindent $\bullet$ We will often deal with two-variable functions of the type
\begin{equation}\label{f}%
f:\left(  x,y\right)  \in\mathcal{C}_{\Omega}\rightarrow f\left(  x,y\right)
\in{\mathbb{R}}
\end{equation}
defined on the cylinder  $\mathcal{C}%
_{\Omega}:=\Omega\times\left(  0,+\infty\right)  $, and measurable with respect to
$x.$ In that case, it will be convenient to define the so-called {\sl Steiner symmetrization} of $\mathcal{C}_{\Omega}$ with
respect to the variable $x$, namely the set \ \hbox{$\mathcal{C}_{\Omega}^{\#}:=\Omega
^{\#}\times\left(  0,+\infty\right).$} Furthermore, we will denote by $\mu
_{f}\left(  k,y\right)  $ and $f^{\ast}\left(  s,y\right)  $ the distribution
function and the decreasing rearrangements of (\ref{f}), with respect to $x$
for $y$ fixed, and we will also define the function%
\[
f^{\#}\left(  x,y\right)  =f^{\ast}(\omega_{N}|x|^{N},y)
\]
which is called the \emph{Steiner symmetrization} of $f$, with respect to the line
$x=0.$ Clearly, $f^{\#}$ is a spherically symmetric and decreasing function
with respect to $x$, for any fixed $y$.

\noindent $\bullet$ There are  some interesting differentiation formulas which turn out to be very useful in our approach. Typically, they are used
when one wants to get sharp estimates satisfied by the rearrangement $u^{\ast}$ of a solution $u$ to a certain parabolic problem, for in that context
it becomes
crucial to differentiate with respect to the extra variable $y$ under the integral symbol, in the form
\[
\int_{\{u(x,y)>u^{*}(s,y)\}}\frac{\partial u}{\partial y}(x,y)\,dx\,.
\]
For the sake of completeness, we recall here two formulas, of first and second order, available in literature. The following proposition can be
found
in
\cite{Mossino}, and is a generalization of a well-known result by Bandle (see \cite{Bandle}).
\begin{proposition}\label{BANDLE}
Suppose that $f\in H^{1}(0,T;L^{2}(\Omega))$ for some $T>0$. Then $$f^{*}\in H^{1}(0,T;L^{2}(0,|\Omega|))$$ and if \
$|\left\{f(x,t)=f^{*}(s,t)\right\}|=0$ \
for
a.e. $(s,t)\in(0,|\Omega|)\times(0,T)$, the following differentiation formula holds:
\begin{equation}
\int_{f(x,y)>f^{*}(s,y)}\frac{\partial f}{\partial y}(x,y)\,dx=\int_{0}^{s}\frac{\partial f^{*}}{\partial y}(\tau,y)\,d\tau.\label{Rakotoson}
\end{equation}
\end{proposition}

Moreover, the following second order differentiation formula (which was also proved in \cite{ATLD} in a more regular framework) is due to Mercaldo and Ferone (see \cite{MR1649548}):

\begin{proposition}
\label{Ferone-Mercaldo} Let $f\in W^{2,\infty}\left(  \mathcal{C}_{\Omega}\right)  $. Then for almost every $y\in(0,+\infty)$ the following
 differentiation formula holds:
\begin{align*}
\int_{f\left(  x,y\right)  >f^{\ast}\left(  s,y\right)  }\frac{\partial^{2}%
f}{\partial y^{2}}\left(  x,y\right)  dx  &  =\frac{\partial^{2}}{\partial
y^{2}}\int_{0}^{s}f^{\ast}\left(  \tau,y\right)  d\tau-\int_{f\left(
x,y\right)  =f^{\ast}\left(  s,y\right)  }\frac{\left(  \frac{\partial
f}{\partial y}\left(  x,y\right)  \right)  ^{2}}{\left\vert \nabla
_{x}f\right\vert }\,d\mathcal{H}^{N-1}\left(  x\right) \\
&  \!\!\!+\left(  \int_{f\left(  x,y\right)  =f^{\ast}\left(  s,y\right)
}\!\frac{\frac{\partial f}{\partial y}\left(  x,y\right)  }{\left\vert
\nabla_{x}f\right\vert }\,d\mathcal{H}^{N-1}\left(  x\right)  \!\right)
^{2}\!\left(  \!\int_{f\left(  x,y\right)  =f^{\ast}\left(  s,y\right)
}\!\frac{1}{\left\vert \nabla_{x}f\right\vert }\,d\mathcal{H}^{N-1}\left(
x\right)  \!\right)  ^{-1}\!.
\end{align*}

\end{proposition}

%%%%%%%%%%%%%%%%%%%%%%%%%%%%%%%%%%%%%%%%%%%%%%%%%%%

\subsection{Mass concentration}

We will provide estimates of the solutions of our elliptic and parabolic problems in terms of their integrals. For that purpose, the following
definition, taken from   \cite{Vsym82}, is remarkably useful.

\begin{definition}
Let $f,g\in L^{1}_{loc}(\R^{N})$ be two radially symmetric functions on $\R^{N}$. We say that $f$ is less concentrated than $g$, and we write
$f\prec
g$ if for
all $R>0$ we get
\[
\int_{B_{R}(0)}f(x)dx\leq \int_{B_{R}(0)}g(x)dx.
\]
\end{definition}
The partial order relationship $\prec$ is called \emph{comparison of mass concentrations}.
Of course, this definition can be suitably adapted if $f,g$ are radially symmetric and locally integrable functions on a ball $B_{R}$. Besides, if
$f$
and $g$ are locally integrable on a general open set $\Omega$, we say that $f$ is less concentrated than $g$ and we write again $f\prec g$ simply if
$f^{\#}\prec g^{\#}$,  but this extended definition has no use if $g$ is not rearranged.

The comparison of mass concentrations enjoys a nice equivalent formulation if $f$ and $g$ are rearranged, whose proof we refer to  \cite{MR0046395},
\cite{Chong}, \cite{VANS05}:

\begin{lemma}\label{lemma1}
Let $f,g\in L^{1}(\Omega)$ be two rearranged functions on a ball $\Omega=B_{R}(0)$. Then $f\prec g$ if and only if for every convex
nondecreasing
function
$\Phi:[0,\infty)\rightarrow [0,\infty)$ with $\Phi(0)=0$ we have
\begin{equation}
\int_{\Omega}\Phi(f(x))\,dx\leq \int_{\Omega}\Phi(g(x))\,dx.
\end{equation}
This result still holds if $R=\infty$ and $f,g\in L^{1}_{loc}(\R^{N})$ with $g\rightarrow0$ as $|x|\rightarrow\infty$.
\end{lemma}
From this Lemma it easily follows that if $f\prec g$ and $f,g$ are rearranged, then
\begin{equation}
\|f\|_{L^{p}(\Omega)}\leq \|g\|_{L^{p}(\Omega)}\quad \forall p\in[1,\infty].
\end{equation}

%%%%%%%%%%%%%%%%%%%%%%%%%%%%%%%%%%%%%%%%%%%%

\section{Elliptic Problems with lower order term}\label{sect.ell}
\setcounter{equation}{0}

\subsection{Recall of existence, uniqueness and main properties}

As explained in the Introduction, the implicit time discretization scheme directly connects the analysis of the evolution equation \eqref{nolin.parab} to solving the elliptic equation \eqref{nolin.ell}. Therefore, we start our analysis by the following nonlocal Dirichlet problem with homogeneous boundary
condition:
\begin{equation} \label{eq.1}
\left\{
\begin{array}
[c]{lll}%
\left(  -\Delta\right)^{\sigma/2}v+  B(v)=f\left(  x\right)   &  & in\text{ }%
\Omega,\\[6pt]
v=0 &  & on\text{ }\partial\Omega,
\end{array}
\right. %
\end{equation}
where $\Omega$ is an open bounded set of ${\mathbb{R}}^{N}$,  $\sigma\in(0,2)$ and $f$ is an integrable function defined in $\Omega$ (we will also take $\Omega={\mathbb{R}}^{N}$, and then no boundary condition is assumed, see below). We assume that the nonlinearity is given by a function $B:\R_{+}\rightarrow\R_{+}$ which is  smooth and monotone increasing with $B(0)=0$ and $B'(v) >0$. It is not essential to consider negative values for our main results, but the general theory can be done in that greater generality, just by assuming that $B$ is extended to a function
$B:\R_{-}\rightarrow\R_{-}$ by
symmetry, $B(-v)=-B(v)$. The fractional-Laplacian operator $(-\Delta)^{\sigma/2}$  acts on functions $u$ in $\Omega$ and is defined through the
spectral
decomposition of $u$, in terms of eigenvalues and eigenfunctions of the Laplacian $-\Delta$ with homogeneous boundary conditions.  Note that we have
changed a bit the notation with respect to equation \eqref{nolin.ell} in the introduction, by eliminating the constant $h>0$, but the change is
inessential for the comparison results.

As explained in   \cite{CT} and \cite{Colorado},  when working in a bounded domain  the fractional Laplacian $(-\Delta)^{\sigma/2}$ can still be defined as a Dirichlet-to-Neumann map (in the same flavor of the construction in  \cite{CaffS} for $\Omega=\ren$), and this allows to connect nonlocal
problems involving $(-\Delta)^{\sigma/2}$ to suitable  degenerate-singular, local problems defined in one more space dimension. In our case, a solution to problem (\ref{eq.1}) is defined as the trace of a properly defined Dirichlet-Neumann problem as follows. If $w$ is a weak solution to the local problem
\begin{equation}
\left\{
\begin{array}
[c]{lll}%
-\operatorname{div}_{x,y}\left(  y^{1-\sigma}\nabla w\right)  =0 &  & in\text{
}\mathcal{C}_{\Omega},\\[6pt]
\ w=0 &  & on\text{ }\partial_{L}\mathcal{C}_{\Omega},\\[6pt]
\displaystyle{-\frac{1}{\kappa_{\sigma}}\lim_{y\rightarrow0^{+}}y^{1-\sigma}\,\dfrac{\partial w}{\partial y}(x,y)}+\,B(w(x,0))=f\left(  x\right)   &
&
in\text{ }\Omega,
\end{array}
\right.  \label{eq.3}%
\end{equation}
where $\mathcal{C}_{\Omega}:=\Omega\times\left(  0,+\infty\right)  $ is the
cylinder of basis $\Omega$,  $\partial_{L}\mathcal{C}_{\Omega}%
:=\partial\Omega\times\lbrack0,+\infty)$ is its lateral boundary, and $\kappa_{\sigma}$ is the constant (see \cite{CaffS})
\[
\kappa_{\sigma}:=\frac{2^{1-\sigma}\,\Gamma(1-\frac{\sigma}{2})}{\Gamma(\frac{\sigma}{2})},
\]
then the trace of $w$ over $\Omega$, $\text{Tr}_{\Omega}(w)=w(\cdot,0)=:v$ is said a solution to
problem (\ref{eq.1}). To make this more precise, we introduce the concept of weak solution to problem \eqref{eq.3}. It is convenient to define the
weighted energy space
$$
X_{0}^{\sigma/2}(\mathcal{C}_{\Omega})=\left\{  w\in H^{1}(\mathcal{C}_{\Omega
}),\,w=0\,\text{ on }\partial_{L}\mathcal{C}_{\Omega}\,\,:\int_{\mathcal{C}%
_{\Omega}}y^{1-\sigma}|\nabla_{x,y} w(x,y)|^{2}\,dxdy<\infty\right\}\,,
$$
equipped with the norm
\begin{equation}
\Vert w\Vert_{X_{0}^{\sigma/2}}:=\left(  \int_{\mathcal{C}_{\Omega}}y^{1-\sigma}\,|\nabla w(x,y)|^{2}\,dxdy\right)  ^{1/2}.\label{norm}
\end{equation}
Then, following \cite{pqrv}, \cite{pqrv2} we provide the following definition

\begin{definition}\label{definition2}
Let $\Omega$ be an open bounded set of $\ren$ and $f\in L^{1}(\Omega)$. We say that $w\in X_{0}^{\sigma/2}(\mathcal{C}_{\Omega})$ is a weak solution
to
\eqref{eq.3}
if $Tr_{\Omega}(B(w))=:B(w(x,0))\in L^{1}(\Omega)$  and
\begin{equation}
\int_{\mathcal{C}_{\Omega}}y^{1-\sigma}\nabla_{x,y} w\cdot \nabla_{x,y}\varphi\,dx\,dy+\int_{\Omega} B(w(x,0))
\,\varphi(x,0)dx=\kappa_{\sigma}\int_{\Omega}f(x)\,\varphi(x,0)dx\label{weakfor}
\end{equation}
for all the test functions $\varphi\in C^{1}(\overline{\mathcal{C}_{\Omega}})$ vanishing on the lateral boundary
$\partial_{L}\mathcal{C}_{\Omega}$.
\end{definition}
If $w$ is a solution to the ``extended problem'' \eqref{eq.3}, then the  trace function $v=\text{Tr}_{\Omega}w$ will be called a weak solution
to
problem
\eqref{eq.1}.
Concerning existence of solutions, their smoothness and $L^{1}$ contraction properties, we excerpt some known results from \cite{pqrv}, \cite{pqrv2}, which can be
extended for our more general nonlinearity $B$.
%%%%%%%%%%%%%%%%%%%%%%%%%%%%%%%%%%%%%%%%%
\begin{theorem}[see \cite{pqrv}]\label{th.exist}
For any $f\in L^{\infty}(\Omega)$ there exists a unique weak solution $w\in X_{0}^{\sigma/2}(\mathcal{C}_{\Omega})$ to problem \eqref{eq.3}, such that
$\text{Tr}_{\Omega}(B(w))\in L^{\infty}(\Omega)$. Moreover,

\smallskip

\noindent {\rm (i)} Regularity: we have \nc  $w\in C^{\alpha}(\mathcal{C}_{\Omega})$ for every $\alpha<\sigma$ if $\sigma\le 1$ (resp. $w\in C^{1,\alpha}(\mathcal{C}_{\Omega})$ for every $\alpha<\sigma-1$ if $\sigma> 1$). Arguing as in {\rm \cite{CT}}, higher regularity of $w$ depends easily on higher regularity of $f$ and $B$.

\noindent {\rm (ii)} $L^{1}$ contraction: if \, $w,\widetilde{w}$  are the solutions to \eqref{eq.3} corresponding to data $f,\widetilde{f}$, the
following
$L^{1}$
contraction property holds:
\begin{equation}
\int_{\Omega}\left[B(w(x,0))-B(\widetilde{w}(x,0))\right]_{+}dx\leq\int_{\Omega}[f(x)-\widetilde{f}(x)]_{+}dx.\label{contraction}
\end{equation}
In particular, we have that $w\geq0$ in $\overline{\mathcal{C}}_{\Omega}$ whenever $f\geq0$ on $\Omega$. Furthermore, if we put
$u:=B(w(\cdot,0))$,
then for
all
$p\in [1, \infty]$
we have
\[
\|u\|_{L^{p}(\Omega)}\leq \|f\|_{L^{p}(\Omega)}.
\]

\noindent {\rm (iii)} For data $f\in L^1(\Omega)$ the weak solution is obtained as the limit of the solutions of approximate problems with $f_n\in L^1(\Omega)\cap L^\infty(\Omega)$, $f_n\to f$ in $L^1$, since then the sequence $\{B(w_n(x,0))\}_n$ also converges in $L^1$ to some $B(w(x,0)$,  and
$\|B(w(x,0))\|_1\le \|f\|_1$, hence $v_n$ in uniformly bounded in $L^p$ for all small $p$. Property (ii) holds for such limit solutions.
%%%%%%%%%%%%%%%%%%
\end{theorem}
%%%%%%%%%%%%%%%%%%%%%%%%%%

\noindent We give a short account of the proof of these results for the reader's convenience. See more details on this issue in the forthcoming work \cite{pqrv4}. In order to get the existence of a weak solution, we first define the integral function of $B$
\[
G(t)=\int_{0}^{t}B(\xi)d\xi\,,
\]
then we minimize the functional
\[
\mathcal{J}(w)=\frac{1}{2\kappa_{\sigma}}\int_{\mathcal{C}_{\Omega}}y^{1-\sigma}\,|\nabla
w|^{2}dx\,dy+\int_{\Omega}G(|w(x,0)|)dx-\int_{\Omega}f(x)\,w(x,0)dx
\]
over the space $X_{0}^{\sigma/2}(\mathcal{C}_{\Omega})$.
where $A:=B^{-1}$ is the inverse of $B$. Arguing as in \cite{pqrv2}, and using the H\"{o}lder, trace, and Young inequalities we find that the
functional $\mathcal{J}$ is coercive on $\mathcal{X}$. In order to prove that $\mathcal{J}$ is weak lower semi-continuous, let
$\left\{w\right\}_{n}$ be a sequence in $X_{0}^{\sigma/2}(\mathcal{C}_{\Omega})$ converging weakly to $w$. By the trace embedding theorem, we have
(up
to subsequences)
\[
w_{n}(\cdot,0)\rightarrow w(\cdot,0)\quad\text{strong in }L^{q}(\Omega)\,\,\,\forall q\in[1,2N/(N-\sigma)),
\]
then
\[
w_{n}(\cdot,0)\rightarrow w(\cdot,0)\quad a.e.\,in\,\Omega.
\]
Now Fatou's lemma implies that $\mathcal{J}$ is weakly lower semicontinuous on $X_{0}^{\sigma/2}(\mathcal{C}_{\Omega})$. Then there exists a
minimizer
$w\in X_{0}^{\sigma/2}(\mathcal{C}_{\Omega})$ of
$\mathcal{J}$. Furthermore, a truncation argument shows that we can suppose $w(\cdot,0)\in L^{\infty}(\Omega)$ and
\[
\|w(\cdot, 0)\|_{L^{\infty}(\Omega)}\leq A(\|f\|_{L^{\infty}(\Omega)}).
\]
Finally, computing the first variation of
$\mathcal{J}$ in
the
direction of any $\varphi\in X_{0}^{\sigma/2}(\mathcal{C}_{\Omega})$ we obtain that $w$ is a weak solution to \eqref{eq.3} in the sense of definition
\eqref{definition2}. The contraction property
\eqref{contraction}
in
Theorem \ref{th.exist} follows by the arguments of \cite{pqrv}-\cite{pqrv2}. \nc

\medskip
%%%%%%%%%%%%%%%%%%%%%%%%%%%%%%%%%%%%%%%%%%%

\noindent $\bullet$ Let us now consider problem \eqref{eq.1} in the whole $\ren$, where the fractional Laplacian is defined by a singular integral. The problem is
\begin{equation} \label{whole}
\left\{
\begin{array}
[c]{lll}%
\left(  -\Delta\right)  ^{\sigma/2}v+  B(v)=f\left(  x\right)   &  & in\text{ }%
\R^{N}\\
&  & \\
v(x)\rightarrow0 &  & as\text{ }|x|\rightarrow\infty,
\end{array}
\right. %
\end{equation}
where $f\in L^{1}(\ren)\cap L^{\infty}(\ren)$, and we can define again a suitable meaning of weak solution, making use of a proper extension problem.
Indeed, if
we denote by $X^{\sigma/2}(\mathcal{C}_{\R^{N}})$, being $\mathcal{C}_{\R^{N}}:=\R_{+}^{N+1} $ the upper half-space, the completion of
$C^{\infty}(\overline{\mathcal{C}_{\R^{N}}})$ with respect to the norm \eqref{norm} with $\Omega$ replaced by $\R^{N}$, then a solution $v$ to
\eqref{whole} is
the trace on $\ren$ of a weak solution
$w\in
X^{\sigma/2}(\mathcal{C}_{\R^{N}})$ to the problem
\begin{equation}
\left\{
\begin{array}
[c]{lll}%
-\operatorname{div}_{x,y}\left(  y^{1-\sigma}\nabla w\right)  =0 &  & in\text{
}\R^{N}\times(0,+\infty)\\[6pt]
\displaystyle{-\frac{1}{\kappa_{\sigma}}\lim_{y\rightarrow0^{+}}y^{1-\sigma}\,\dfrac{\partial w}{\partial y}(x,y)}+\,B(w(x,0))=f\left(  x\right)   &
&
x\in\R^{N}.
\end{array}
\right.  \label{eq.6}%
\end{equation}
Of course, we mean that $w\in X^{\sigma/2}(\mathcal{C}_{\R^{N}})$ is a weak (energy) solution to \eqref{eq.6} if equality \eqref{weakfor} holds, with
$\Omega$
replaced by
$\ren$. In order to obtain the existence and uniqueness of solution to problem \eqref{eq.6}, we can ague as in \cite{pqrv2}. For any $R>0$, we
consider
the
solution $w_{R}$ to \eqref{eq.3} corresponding to the data $f_{R}=f\chi_{B_{R}(0)}$, where $\Omega$ is the ball $B_{R}(0)$ centered at the
origin. If
the data
$f$ is nonnegative (the case of changing sign data can be treated as in \cite{pqrv}), we obtain an increasing sequence of non-negative solutions
$\left\{w_{R}\right\}$ converging to a weak solution $w$ to the problem \eqref{eq.6} in the upper half-space. Then the contraction property
\eqref{contraction}
holds in $\ren$, from which uniqueness and preserving sign property follow. Moreover, if \ $u=B(w(\cdot,0))$ \ then we have
\[
\|u\|_{L^{1}(\ren)}\leq \|f\|_{L^{1}(\ren)},\quad
\|u\|_{L^{\infty}(\ren)}\leq \|f\|_{L^{\infty}(\ren)}.
\]

\noindent{\bf Remarks.} The approximation method we have used to prove the comparison theorem in the whole $\R^{N}$ actually says that we can
approximate the solution $v$ to problem \eqref{whole} with the fractional Laplacian on $\R^{N}$ by a sequence of solutions of Dirichlet problems of
the type \eqref{eq.1} with the
fractional
laplacian
defined on balls, with homogeneous boundary data.

 We point out that Theorem \ref{th.exist} and the related considerations of existence of solutions on $\ren$ still hold if
$B:\R\rightarrow\R$
is assumed to be increasing and $B(0)>0$ (and this remark will enter in Subsection 3.5). If we want to extend $B$ to the whole real axis it
suffices to
set $B(-v)=2B(0)-B(v)$ for all $v\geq0$.

\medskip

From now on, we will always assume that the right-hand side $f$ is nonnegative.

%%%%%%%%%%%%%%%%%%%%%%%%%%%%%%
\subsection{The extended problem and concentration comparison}

Let us address the comparison issue. Our goal here is to compare the solution $v$ to \eqref{eq.1} with the solution $V$ to the problem
\begin{equation}
\left\{
\begin{array}
[c]{lll}%
\left(  -\Delta\right)  ^{\sigma/2}V+\, B(V)=f^{\#}\left(  x\right)   &  & in\text{ }%
\Omega^{\#}\\[6pt]
V=0 &  & on\text{ }\partial\Omega^{\#}.
\end{array}
\right.  \label{eq.4}%
\end{equation}
A reasonable way to do that is to compare the solution $w$ to \eqref{eq.3} with the solution $\psi$ to the problem
\begin{equation}
\left\{
\begin{array}
[c]{lll}%
-\operatorname{div}_{x,y}\left(  y^{1-\sigma}\nabla \psi\right)  =0  &  & in\text{
}\mathcal{C}_{\Omega^{\#}}\\[6pt]
\psi=0 &  & on\text{ }\partial_{L}\mathcal{C}_{\Omega^{\#}}\\[6pt]
\displaystyle{-\frac{1}{\kappa_{\sigma}}\lim_{y\rightarrow0^{+}}y^{1-\sigma}\,\dfrac{\partial \psi}{\partial y}(x,y)}+\,B(\psi(x,0))=f^{\#}\left(
x\right)
&  & in\text{ }\Omega^{\#},
\end{array}
\right.  \label{eq.5}%
\end{equation}
where $\psi(x,0)=V(x)$.
According to \cite{BV}, using the change of variables
\[
z=\left(  \frac{y}{\sigma}\right)  ^{\sigma},
\]
problems \eqref{eq.3} and \eqref{eq.5} become respectively

\begin{equation}
\left\{
\begin{array}
[c]{lll}%
-z^{\nu}\dfrac{\partial^{2}w}{\partial z^{2}}-\Delta_{x}w=0 &  & in\text{
}\mathcal{C}_{\Omega}\\
&  & \\
w=0 &  & on\text{ }\partial_{L}\mathcal{C}_{\Omega}\\
&  & \\
-\dfrac{\partial w}{\partial z}\left(  x,0\right)  =%
\,\sigma^{\sigma-1}\kappa_{\sigma}\left(f\left(  x\right)-B(w(x,0))\right)  &  & in\text{ }\Omega,
\end{array}
\right.  \label{eq.23}%
\end{equation}
and
\begin{equation}
\left\{
\begin{array}
[c]{lll}%
-z^{\nu}\dfrac{\partial^{2}\psi}{\partial z^{2}}-\Delta_{x}\psi=0 &  & in\text{
}\mathcal{C}_{\Omega}^{\#}\\
&  & \\
\psi=0 &  & on\text{ }\partial_{L}\mathcal{C}_{\Omega}^{\#}\\
&  & \\
-\dfrac{\partial \psi}{\partial z}\left(  x,0\right)  =%
\,\sigma^{\sigma-1}\kappa_{\sigma}\left(f^{\#}\left(  x\right)-B(\psi(x,0))\right)  &  & in\text{ }\Omega^{\#}.
\end{array}
\right.  \label{eq.24}%
\end{equation}
where $\nu:=2\left(  \sigma-1\right)  /\sigma.$ \
Then, the problem reduces to prove the concentration comparison between the solutions $w(x,z)$ and $\psi(x,z)$ to \eqref{eq.23}-\eqref{eq.24}
respectively.
Following
\cite{BV}, using standard symmetrization tools (among which the differentiation formulas in Propositions \eqref{BANDLE}-\eqref{Ferone-Mercaldo} are
essential), if
we
introduce the
function
\begin{equation}
{Z}(s,z)=\int_{0}^{s}(w^{\ast}(\tau,z)-\psi^{\ast}(\tau,z))d\tau\,,
\end{equation}
then we get the inequality
\begin{equation}
-z^{\nu}{Z}_{zz}-p\left(  s\right)
{Z}_{ss}\leq0\label{symineq}
\end{equation}
for a.e. $(s,z)\in D:=\left(  0,|\Omega| \right)
\times\left(  0,+\infty\right)  $ . Obviously, we have
\begin{equation}
Z(0,y)={Z}_{s}(|\Omega|,y)=0\label{boundcond}.
\end{equation}
A crucial point in our arguments below is played by the derivative of $Z$ with respect to $z$. Due to the boundary conditions contained
in
\eqref{eq.23}-\eqref{eq.24}, we have
\begin{equation}\label{Z_yboundary.formula}
{Z}_{z}(s,0)\geq \theta_{\sigma}\int_{0}^{s} (B(w^*(\tau,0))-B(\psi^{\ast}(\tau,0))\, d\tau
\end{equation}
where
\[
\theta_{\sigma}:=\sigma^{\sigma-1}\kappa_{\sigma}.
\]
Now observe that the function $$Y(s,0)=\int_{0}^{s}B(w^*(\tau,0))-B(\psi^*(\tau,0))\, d\tau$$
has the same points of maximum or minimum and the same regions of monotonicity than ${Z}(s,0)$.

%%%%%%%%%%%%%%%%%%%%%%%%%%%%%%%%%%%%%%%%%%%%%%%%%%%%%%%%%%%%%%%%%%%%

\subsection{Comparison result for concave $B$}

\begin{theorem}\label{thm.ell.concave} Let $v$ be the nonnegative solution of  problem $\eqref{eq.1}$ posed in a bounded domain with zero
Dirichlet
boundary
condition, nonnegative data $f\in L^1(\Omega)$ and nonlinearity $B(v)$ given by a concave function with $B(0)=0$ and $B'(v)>0$ for all $v>0$. If $V$
is the solution of the corresponding symmetrized problem, we have
\begin{equation}
v^\#(x)\prec V(x).
\end{equation}
The same is true if $\Omega=\ren$.
\end{theorem}

\noindent {\sl Proof.} In this case we pose the problem first in a bounded domain $\Omega$ of $\ren$ with smooth boundary. We also assume that $f$ is smooth, bounded and compactly supported, since the comparison result for general data can be obtained later by approximation using the $L^1$ dependence of the map $f\mapsto B(v)$.

(i) We want to prove that  $Z(s,0)\le 0$ for all $s\in [0,|\Omega|]$. It is easy to prove that a positive maximum of $Z(s,z)$ cannot happen at
the
lateral
boundary $s=|\Omega|$ for $z>0$ by the stated boundary conditions and the Hopf's boundary principle.

In order to study the possible positive maximum at the line $z=0$ we proceed as follows.
The concavity of $B$ implies that for $a,b\ge 0$ we have $B(a)-B(b)\ge B'(a)(a-b)$. Using this and
\eqref{Z_yboundary.formula}, it follows that
\begin{align}
& Z_{z}(s,0)\geq \theta_{\sigma}Y(s,0)\nonumber
\geq
\theta_{\sigma}\int_{0}^{s}B^{\prime}(v^{\ast}(\tau))[w^{\ast}(\tau,0)-\psi^{\ast}(\tau,0)]d\tau\nonumber
\\&=\theta_{\sigma}\int_{0}^{s}B^{\prime}(v^{\ast}(\tau))Z_{s}(\tau,0)d\tau.\label{dery}
\end{align}
If we set $g(s):=B^{\prime}(v^{\ast}(s))$, since $B^{\prime}$ is decreasing,  we notice that $g$ is an increasing function bounded from
below by
$g(0)=B^{\prime}(\|v\|_{\infty})$. If the positive maximum of $Z(s,0)$ happens for $s=s_0$ then, using an integration by parts in \eqref{dery} we
can
write
\begin{align}
&Z_{z}(s_{0},0)\geq\theta_{\sigma}\int_{0}^{s_0}B^{\prime}(v^{\ast}(s))Z_{s}(s,0)ds\nonumber\\
&=\theta_{\sigma}\left[g(0)Z(s_{0},0)+\int_{0}^{s_0}[Z(s_{0},0)-Z(s,0)]dg(s)\right]>0\label{relmaxZY}
\end{align}
which is impossible because $Z_{z}(s_{0},0)\leq0$. Then $Z(s_0,0)\leq0$, that is $Z\leq0$, namely
\[
\int_{0}^{s}w^{\ast}(\tau,z)\,d\tau\leq\int_{0}^{s}\psi^{\ast}(\tau,z)\,d\tau.
\]

Another remark is that either $Z\equiv 0$ or
\begin{equation}
Z<0 \text{ in } (0,|\Omega|)\times[0,\infty)\label{negativity}:
\end{equation}
indeed, if $Z\not\equiv0$ for the previous arguments it cannot be $Z=0$ in some points of $(0,|\Omega|)\times(0,\infty)$ (otherwise it would reach
the
maximum in
this domain, hence it would be constantly 0 by the maximum principle). On the other hand, if $Z(s_{0},0)=0$ for some point $s_{0}\in
(0,|\Omega|)$, by
the Hopf
boundary maximum principle we have $Z_{z}(s_{0},0)<0$, but by \eqref{relmaxZY} we have $Z_{z}(s_{0},0)\geq0$.

\medskip

(ii) Here is a simpler proof in the important special case of the linear fractional diffusion,
i.\,e., when $B(v)=\,v$. Indeed, from \eqref{Z_yboundary.formula} we have the inequality
\[
Z_{z}(s,0)\geq\theta_{\sigma} Z(s,0).\nonumber\\
\]
Now we simply observe that \eqref{symineq} can be rewritten as
\[
-p(s)^{-1}Z_{zz}-z^{-\nu}
Z_{ss}\leq0
\]
therefore, multiplying both sides by $Z_{+}$ and integrating by parts over the strip $[0,|\Omega|]\times(0,+\infty)$, the boundary conditions
\eqref{boundcond}
and the fact that $Z(s,z)\rightarrow0$ as $z\rightarrow\infty$ imply
\begin{align*}
&\int_{0}^{|\Omega|}p(s)^{-1}Z_{z}(s,0)Z_{+}(s,0)ds+\int_{0}^{\infty}\int_{0}^{|\Omega|}z^{-\nu}|\left(Z_{+}\right)_{s}|^{2}ds\,dz\\
&+\int_{0}^{\infty}\int_{0}^{|\Omega|}
p(s)^{-1}|\left(Z_{+}\right)_{z}|^{2}ds\,dz\leq0
\end{align*}
namely
\[
\int_{0}^{\infty}\int_{0}^{|\Omega|}z^{-\nu}|\left(Z_{+}\right)_{s}|^{2}ds\,dz
+\int_{0}^{\infty}\int_{0}^{|\Omega|}
p(s)^{-1}|\left(Z_{+}\right)_{z}|^{2}ds\,dz\leq0.
\]
hence $Z_{+}\equiv0$.

\medskip

\noindent $\bullet$ {\sl Problem in the whole space.} The previous arguments still apply if the problem is posed in the whole space $\R^{N}$, namely
if
$v$
solves
\eqref{whole}. Indeed, in this case we may use the boundary condition $Z_{s}(s,y)\rightarrow0$ as $s\rightarrow\infty$. Alternatively, according
to
what
remarked
in Section 3, we may approximate the solution $w$ to the elliptic problem in the upper half-space \eqref{eq.6},
with nonnegative $f\in L^{1}(\ren)\cap L^{\infty}(\ren)$, with the family $w_R$ of solutions to problems of the type
\begin{equation}
\left\{
\begin{array}
[c]{lll}%
-\operatorname{div}_{x,y}\left(  y^{1-\sigma}\nabla w_{R}\right)  =0 &  & in\text{
}\mathcal{C}_{B_{R}},\\[6pt]
w_{R}=0 &  & on\text{ }\partial_{L}\mathcal{C}_{B_{R}},\\[6pt]
\displaystyle{-\frac{1}{\kappa_{\sigma}}\lim_{y\rightarrow0^{+}}y^{1-\sigma}\,\dfrac{\partial w_R}{\partial y}(x,y)}+\,B(w_{R}(x,0))=f_{R}\left(
x\right)
&
& in\text{ }B_{R},
\end{array}
\right.  \label{eq.7}%
\end{equation}
where $B_{R}$ is a ball of radius $R$ at the origin. According to Theorem \ref{thm.ell.concave}, we obtain
\begin{equation}
\int_{0}^{s}w_{R}^{\ast}(\tau,y)d\tau\leq\int_{0}^{s}\psi_{R}^{\ast}(\tau,y)d\tau\label{compball}
\end{equation}
for all $s\in[0,|B_{R}|]$ and $y\geq0$ where $\psi_{R}$ is the solution to
\begin{equation*}
\left\{
\begin{array}
[c]{lll}%
-\operatorname{div}_{x,y}\left(  y^{1-\sigma}\nabla \psi_{R}\right)  =0  &  & in\text{
}\mathcal{C}_{B_{R}}\\[6pt]
\psi_{R}=0 &  & on\text{ }\partial_{L}\mathcal{C}_{B_{R}}\\[6pt]
\displaystyle{-\frac{1}{\kappa_{\sigma}}\lim_{y\rightarrow0^{+}}y^{1-\sigma}\,\dfrac{\partial \psi_R}{\partial
y}(x,y)}+\,B(\psi_{R}(x,0))=f_{R}^{\#}\left(
x\right)   &  & in\text{ }B_{R}.
\end{array}
\right.  \label{eq.8}%
\end{equation*}
 Then we get (see Theorem 7.3 in \cite{pqrv2}) $w_{R}\rightarrow w$ and  $\psi_{R}\rightarrow \psi$ as $R\rightarrow\infty$, where $\psi$ solves
\begin{equation}
\left\{
\begin{array}
[c]{lll}%
-\operatorname{div}_{x,y}\left(  y^{1-\sigma}\nabla \psi\right)  =0 &  & in\text{
}\R^{N}\times(0,+\infty)\\[6pt]
\displaystyle{-\frac{1}{\kappa_{\sigma}}\lim_{y\rightarrow0^{+}}y^{1-\sigma}\,\dfrac{\partial \psi}{\partial y}(x,y)}+\,B(\psi(x,0))=f^{\#}\left(
x\right)
&  & x\in\R^{N}.
\end{array}
\right.  \label{eq.10}%
\end{equation}
Therefore, letting $R\rightarrow\infty$ in \eqref{compball} we find
\begin{equation*}
\int_{0}^{s}w^{\ast}(\tau,y)d\tau\leq\int_{0}^{s}\psi^{\ast}(\tau,y)d\tau\label{eq.11}
\end{equation*}
for all $s\geq0$ and $y\geq0$.

\medskip

\noindent {\bf Remark.}  We also wanted to prove that $Y(s,0)\le 0$, i.\,e.,
\begin{equation}
\int_{0}^{s}B(v^{\ast}(\tau))d\tau \leq\int_{0}^{s}B(V^{\ast}(\tau))d\tau.\label{comparison}
\end{equation}
but it did not work. See next section.

%%%%%%%%%%%%%%%%%%%%%%%%%%%%%%%%%%%%%%%%%%%%%%%%%%%%%%%%%%%%

\subsection{Comparison of concentrations for radial problems}

This second result is a variation and extension of the previous comparison result. We consider the same assumptions on $B$ and  $\Omega$.

\begin{theorem}\label{thm.ell.concave.rad} Let $v_1, v_2$ be two nonnegative solutions of  problem $\eqref{eq.1}$ posed in a ball $B_{R}(0)$,
with
$R\in(0,+\infty]$ with zero Dirichlet boundary conditions if $R<+\infty$, nonnegative radially symmetric decreasing data $f_1, f_2\in L^1(B_{R}(0))$
and
nonlinearity
$B(v)$
given by a concave function for $v\ge0$, with $B(0)=0$ and $B'(v)>0$ for all $v>0$. Then $v_1$ and $v_2$ are rearranged, and
\begin{equation}
 f_1\prec f_2 \quad \mbox{implies} \quad v_1\prec v_2\,.
\end{equation}
\end{theorem}
\noindent {\sl Proof.} As in the proof of Theorem \ref{thm.ell.concave}, we arrive at the inequality
\begin{equation}
-z^{\nu}Z_{zz}^{1,2}-p\left(  s\right)
Z_{ss}^{1,2}\leq0\label{symineq2}
\end{equation}
satisfied a.e. in the strip $(0,|B_{R}(0)|)\times(0,+\infty)$ by the function
\[
Z^{1,2}(s,z)=\int_{0}^{s}(w_{1}^{\ast}-w_{2}^{\ast})d\tau
\]
where $w_1$ and $w_2$ are the solutions of the extensions problems associated to $v_1$ and $v_{2}$ respectively. Concerning the boundary
conditions,
since we
have $f_1\prec f_2$ we get
\begin{align*}
&Z_{y}^{1,2}(s,0)\geq\theta_{\sigma}\int_{0}^{\tau} (B(w_{1}^*(\tau,0))-B(w_{2}^{\ast}(\tau,0))\,
d\tau+\int_{0}^{s}\left(f^{\ast}_{2}-f^{\ast}_{1}\right)d\tau\\
&\geq \int_{0}^{s} (B(w_{1}^*(\tau,0))-B(w_{2}^{\ast}(\tau,0))d\tau.
\end{align*}
Then we conclude as in the proof of Theorem \ref{thm.ell.concave}.

%%%%%%%%%%%%%%%%%%%%%%%%%%%%%%%%%%%%%%%%%%%%%%%%%%%%%%%%%%%%%%%%%%%%%%%
\subsection{Comparison results for convex $B$}

In the case of a convex nonlinearity we prove a stronger result in the whole space. For the sake of clarity, we  first prove the result when $B$
is a
\emph{superlinear} nonlinearity in the sense that is made precise next:

\begin{theorem}\label{thm.ell.convex}   Let $v$ be the nonnegative solution of  problem $\eqref{whole}$ posed in $\Omega=\ren$, nonnegative data
$f\in
L^1(\ren)$
and nonlinearity given by a convex function $B:\R_{+}\rightarrow\R_{+}$ which is smooth, and superlinear:  $B(v)\geq \varepsilon v$ for some
$\varepsilon>0$ and all $v\geq0$. If $V$ is the solution of the corresponding symmetrized problem, we have
\begin{equation}
v^\#\prec V, \qquad B(v^\#) \prec B(V).
\end{equation}
\end{theorem}

\noindent {\bf Remark.} The simplest example of superlinear nonlinearity is of course the linear case, $B(v)=cv$. A nontrivial  example  from the
literature would be the remarkable nonlinearity $A(t)=\log(1+t)$ in the model of logarithmic diffusion, \cite{pqrv3}. Then, $B(s)=A^{-1}(s)=e^{s}-1$,
$s\geq0$. We will  relax the restriction $B(v)\geq \varepsilon v$ below by approximation.

\medskip

\noindent {\sl Proof.}
In order to prove that $Z(s,0)\le 0$ for all $s\in [0,\infty)$ we argue as follows. We have
$ B(w^*(\tau,0))-B(\psi^*(\tau,0)=B'(\xi)(w^*(\tau,0))-\psi^*(\tau,0))$, where $\xi $ is an intermediate value between $w^*(\tau,0)$
and
$\psi^*(\tau,0)$. Since $B$ is convex, $B'$ is an increasing real function and
$$
B(w^*(\tau,0))-B(\psi^*(\tau,0))\le B'(w^*(\tau,0)))(w^*(\tau,0))-\psi^*(\tau,0))
$$
Due to the maximum principle and the boundary conditions \eqref{boundcond}, unless $Z$ is constant, the maximum of $Z$ can be achieved either on
the
half-line
$\left\{(0,z):z\geq0\right\}$ or on the segment line $\left\{(s,0):s\in[0,\infty)\right\}$. Suppose this second circumstance occurs, and let $(s_{0},0)$ be a maximum point. Assume  $s_{0}>0$. We also have $Z_z(s_0,0)<0$ by Hopf's maximum principle, and by \eqref{Z_yboundary.formula}, this
leads to $Y(s_0,0)<0$. Then for $s>s_0$
$$
\begin{array}{l}
\displaystyle Y(s,0)-Y(s_0,0)=\int_{s_0}^s [B(v^*(\tau))- B(V^{\ast}(\tau))]\, d\tau\\
[6pt]
\displaystyle \le \int_{s_0}^s
B'(w^*(\tau,0))(v^*(\tau)-V^{\ast}(\tau))\, d\tau.
\end{array}$$
After integration by parts
$$
\begin{array}{l}
\displaystyle Y(s,0)-Y(s_0,0) \le \left[B'(v^*(\tau))(Z(\tau,0)-Z(s_0,0))\right]_{s_0}^{s} -\\
[6pt]
\displaystyle \int_{s_0}^s B''(v^*(\tau))v^*_s(\tau)(Z(\tau,0)-Z(s_0,0))d\tau.
\end{array}
$$
Since $Z$ has a maximum at $s_0$ and $B'$ is positive, the first term in the RHS is nonpositive. As for the second, we have: $B''>0$, $v^*_s<0$, and $Z(s,0)-Z(s_0,0)\le 0$, hence the last term is also nonpositive. We conclude that $Y(s,0)\le Y(s_0,0)<0$ for all $s>s_0$. This is a contradiction, because by the
conservation of mass property (see proposition \ref{prop.7} below) we have $Y(\infty,0)=0$. Then $s_{0}=0$ and $Z\leq0$.

(ii) Once we have $Z(s,z)\le 0$ we also want to prove that  $Y(s,0)\le 0$. We use the fact that $s=0$ is a point of maximum of $Z$ and  write
$$
Y(s,0) \le \left[B'(v^*(\tau))Z(\tau,0)\right]_{0}^s -
 \int_{0}^s B''(v^*(\tau))v^*_s(\tau)Z(\tau,0))\,d\tau\le 0.
$$
Also, we obtain the same result by using Lemma \ref{lemma1}, taking advantage of the convexity of $B$ and choosing any convex, increasing
function
$\Phi:[0,\infty)\rightarrow[0,\infty)$.
This ends the proof of the concentration comparison theorem in this case. \qed

\medskip

\noindent $\bullet$  The only remaining question here is then to prove that $\|B(w(x,0))\|_{L^1}=\|B(\psi(x,0))\|_{L^1}$.  Under the additional
assumption  $B(s)\ge \varepsilon s$ for all $s>0$, this will be a consequence of the
following mass conservation result for the solutions of the elliptic equation.

\begin{proposition}\label{prop.7} Let $v$ be the weak solution of $ (-\Delta)^{\sigma/2}v+  B(v)=f$ with  $f\in L^1(\ren)$ nonnegative and let
$u=B(v)$, with $B$ satisfying the
same assumptions as in Theorem {\rm \ref{thm.ell.convex}}. Then  we have
$$
\int_{\ren} u(x)\,dx=\int_{\ren} f(x)\,dx.
$$
\end{proposition}

\noindent {\sl Proof.}  Using a nonnegative nonincreasing cutoff function $\zeta(s)$ such that $\zeta(s)=1$ for $0\leq s\leq 1$ and $\zeta(s)=0$
for $s\geq2$, we rescale such function to $\zeta_R(x)=\zeta(|x|/R)$\nc. Then we have
\begin{equation}
\displaystyle \int_{\ren} f(x)\,\zeta_R(x)\,dx-\int_{\ren} u(x)\,\zeta_R(x)\,dx=\int_{\ren} v\,((-\Delta)^{\sigma/2}\zeta_R)\,dx\label{conservation}
\end{equation}
Due to the superlinearity assumption, if $u(x)=B(v(x))$ we get
$$
\left|\int_{\ren} v\,(-\Delta)^{\sigma/2}\zeta_R\,dx\right| \le \frac{1}{\varepsilon}\int_{\ren } |u(x)(-\Delta)^{\sigma/2}\zeta_R|\,dx \le
\frac{c}{R^{\sigma}} \int_{\ren} |u(x)|\,dx
$$
which in the limit $R\to \infty$ tends to zero. \qed

\noindent $\bullet$ The property of mass conservation for solutions in the whole space is also true for some convex $B$ that are not superlinear,
like
$B(v)=v^p$ with some $p>1$ but near 1, but it is
not true for all $p>1$. However, the comparison result we are looking for (which will extend Theorem \ref{thm.ell.convex} for such kind of
nonlinearity) will be true and the proof
proceeds by an approximation process, approximating $B(v)$ by $B_\ve(v)=B(v)+\ve v$, solving the approximate problem, deriving the comparison result
and passing to the limit. The details are as follows.

%%%%%%%%%%%%%%%%%%%%%%%%%%%%%%%%%%%%%%%%%%%%%%%%%%%%%%%%%%%%%%%%%

\begin{proposition}\label{prop.8}
Suppose $f\in L^{1}(\ren)$ and $B:\R_{+}\rightarrow\R_{+}$ is smooth, convex, $B(0)=0$ and $B^{\prime}(v)>0$ for all $v>0$. Let $v_{\ve}$ be the
solution of
problem \eqref{whole}, with nonlinearity given by $B_{\ve}(v)=B(v)+\ve v$. Then $v_{\ve}\rightarrow v$ as $\ve\rightarrow0$ pointwise and in
$L^{1}(\ren)$.
\end{proposition}
\begin{proof}
We first prove the result on a bounded domain $\Omega$. Suppose that $f\in L^{\infty}(\Omega)$, let $v_{\ve}$ be the solution to \eqref{eq.1},
and let
$w_{\ve}$
be its $\alpha-$ harmonic extension to the cylinder $\mathcal{C}_{\Omega}$. We have that $w_{\ve}$ is obtained as minimizer of the functional
\[
\mathcal{J}_{\ve}(w)=\frac{1}{2\kappa_{\sigma}}\int_{\mathcal{C}_{\Omega}}y^{1-\sigma}\,|\nabla
w|^{2}dx\,dy+\int_{\Omega}G_{\ve}(|w(x,0)|)dx-\int_{\Omega}f\,w(x,0)dx
\]
with
\[
G_{\ve}(t)=\int_{0}^{t}\left[B(\xi)+\ve\xi\right]d\xi
\]
over the space $X_{0}^{\sigma/2}(\mathcal{C}_{\Omega})$.  Moreover the trace $v_{\varepsilon}$ over $\Omega$ of $w_{\varepsilon}$ is bounded and
\begin{equation}
\|w_{\varepsilon}(\cdot, 0)\|_{L^{\infty}(\Omega)}\leq A(\|f\|_{L^{\infty}(\Omega)}).\label{limitw}
\end{equation}
Taking $w_{\ve}$ as a test function in the weak formulation of problem \eqref{eq.3},
namely
in the
formula
\begin{equation}
\int_{\mathcal{C}_{\Omega}}y^{1-\sigma}\,\nabla w_{\ve}\cdot \nabla\varphi\,dx\,dy+\int_{\Omega} B(w_{\ve}(x,0))
\,\varphi(x,0)dx+\ve\int_{\Omega}w_{\ve}(x,0)
\,\varphi(x,0)dx=\kappa_{\sigma}\int_{\Omega}f(x)\,\varphi(x,0)dx\label{weakforve}
\end{equation}
for $\varphi\in X_{0}^{\sigma/2}(\mathcal{C}_{\Omega})$, the Young and trace inequalities imply that $\left\{w_{\ve}\right\}$ is bounded in
$X_{0}^{\sigma/2}(\mathcal{C}_{\Omega})$. Then we can extract a subsequence $\left\{w_{\ve}\right\}$ (we used the same labeling for simplicity) such
that
\[
w_{\ve}\rightharpoonup w\quad \text{weak\,in }\,X_{0}^{\sigma/2}(\mathcal{C}_{\Omega}).
\]
Then the compactness of the trace embedding inequality gives
\[
w_{\ve}(\cdot,0)\rightarrow w(\cdot,0)\quad\text{as}\,\ve\rightarrow0\,\text{ strong in }L^{q}(\Omega)\,\,\,\forall q\in[1,2N/(N-\sigma)).
\]
Using \eqref{limitw}, Lebesgue's dominated convergence implies
\[
\int_{\Omega}B(w_{\ve}(\cdot,0))\varphi(\cdot,0)\,dx\rightarrow \int_{\Omega}B(w(\cdot,0))\varphi(\cdot,0)\,dx
\]
for all $\varphi\in X_{0}^{\sigma/2}(\mathcal{C}_{\Omega})$. This is enough to pass to the limit in \eqref{weakforve} and obtain
that
$w$ is the
weak solution to \eqref{eq.3}, that is $v=w(\cdot,0)$ solves \eqref{eq.1}

If $f\in L^{\infty}(\ren)$, let $v$ be the solution to \eqref{whole}. We know that the solution $v_{\ve}$ to \eqref{whole} with nonlinearity
$B_{\ve}$
is the
trace on $\ren$ of the solution $w_{\ve}$ to the problem \eqref{eq.6}, with the nonlinearity $B_{\ve}$. By the arguments we recalled in
Subsection 3.1,
we have
that the sequence of solutions $\left\{w_{\ve}^{R}\right\}_{R>0}$ to problem \eqref{eq.3}, defined on the cylinder $\mathcal{C}_{B_{R}(0)}$,
with
nonlinearity
$B_{\ve}$ and data $f_{R}=f\chi_{B_{R}(0)}$, converges pointwise to $w_{\ve}$ as $R\rightarrow\infty$, that is
\begin{equation}
w_{\ve}^{R}\rightarrow w_{\ve}\quad\text{pointwise as }R\rightarrow\infty.\label{conve1}
\end{equation}
Moreover, if $v_{\ve}^{R}=w_{\ve}^{R}(\cdot,0)$, the arguments explained above show that
\begin{equation}
v_{\ve}^{R}\rightarrow v^{R}\quad\text{as  }\ve\rightarrow0 \text{  strong in }L^{q}(B_{R}(0))\,\,\forall q\in[1,2N/(N-\sigma))\label{conve2}
\end{equation}
where $v^{R}$ is the solution to the Dirichlet problem \eqref{eq.1}, posed on the ball $B_{R}(0)$. In addition, we have
\begin{equation}
v^{R}\rightarrow v\quad\text{pointwise as } R\rightarrow\infty.\label{conve3}
\end{equation}
Then \eqref{conve1},\eqref{conve2},\eqref{conve3} implies that $v_{\ve}$ converges to $v$ pointwise and in $L^{1}(\ren)$.
\end{proof}
%\begin{align}
%&v=\lim_{R\rightarrow\infty}v^{R}=\lim_{R\rightarrow\infty}\lim_{\ve\rightarrow0}v_{\ve}^{R}\\
%&=\lim_{\ve\rightarrow0}\lim_{R\rightarrow\infty}w_{\ve}^{R}(\cdot,0)=\lim_{\ve\rightarrow0}w_{\ve}(\cdot,0)=\lim_{\ve\rightarrow0}v_{\ve}.
%\end{align}

Then we are able to prove the comparison result of Theorem \ref{thm.ell.convex} also for power nonlinearities like $B(v)=v^p$ for large $p>1$ (which
means that we can include the fast diffusion
range when we pass to the parabolic setting in Section \ref{sec.par})
\rm

\begin{theorem}\label{genconvex}
Let $v$ be the nonnegative solution of  problem $\eqref{whole}$, nonnegative data $f\in L^1(\ren)$
and the nonlinearity given by a convex function $B:\R_{+}\rightarrow\R_{+}$, with $B(0)=0$ and $B^{\prime}(v)>0$ for all $v>0$. If $V$ is the
solution
of the corresponding symmetrized problem, the conclusion of Theorem {\rm \ref{thm.ell.convex}} still holds.
\end{theorem}

\begin{proof}
By virtue of Theorem \ref{thm.ell.convex} we have
\begin{equation}
v_{\ve}^\#\prec V_{\ve}, \qquad B(v_{\ve}^\#) \prec B(V_{\ve})\label{compeps}
\end{equation}
for all $\ve>0$, being $v_{\ve}$, $V_{\ve}$ the solution of problem \eqref{whole} and its symmetrized with the nonlinearity $B_{\ve}$. By
Proposition \ref{prop.8} we have that for all $s>0$
\[
\int_{0}^{s}v_{\ve}^{\ast}\,d\tau\rightarrow\int_{0}^{s}v^{\ast}\,d\tau,\quad\int_{0}^{s}V_{\ve}^{\ast}\,d\tau\rightarrow\int_{0}^{s}V^{\ast}\,d\tau
\]
as $\ve\rightarrow0$. Passing to the limit in \eqref{compball} we find the desired result.
\end{proof}

\normalcolor
%%%%%%%%%%%%%%%%%%%%%%%%%%%%%%%%%%%%%%%%%%%%%%%%%%%%%%%%%%%%%%%%%%%%%%%%%%%%

\subsection{Second comparison result for convex $B$} Here is the second result, about comparison of concentrations for radial problems.
We leave the proof to the reader.

\begin{theorem}\label{thm.ell.convex2}  Let $v_1, v_2$ be two nonnegative solutions of  problem $\eqref{eq.1}$ posed in  $\Omega=\ren$, with
nonnegative
radially
symmetric decreasing data $f_1, f_2\in L^1(\Omega)$ and nonlinearity $B(v)$ given by a convex function with $B(0)=0$ and $B'(v)>0$ for all $v>0$.
Then,
$v_1$
and
$v_2$ are rearranged, and for $f_1\prec f_2$ we have
\begin{equation}
v_1(x)\prec v_2(x), \qquad B(v_1)\prec B(v_2)\,.
\end{equation}
\end{theorem}

\noindent {\bf Remark.} These results are in perfect agreement with the results of \cite{Vsym82} for the standard Laplacian case.

\

%%%%%%%%%%%%%%%%%%%%%%%%%%%%%%%%%%%%%%%%%%%%%%%%%%%%%%%%%%%%%%%%%
\section{Counterexample for elliptic concentration comparison with convex powers}\label{sec.ell.count}

If we compare the results of the preceding section for convex $B$ and concave $B$ we realize that the conclusion is weaker in the latter case. This seemed to us a possible defect in the technique since in the case of the standard diffusion $\sigma=2$  the results are identical.But it turned out that in the fractional equation the concave case has an essential difficulty.

Our aim is now to prove that actually \emph{the general concentration comparison does not hold} in the whole $\R^{N}$, for an equation of the form
\begin{equation}\label{eqell.h}
h \, L_{\sigma}(u^m)+u =f, \qquad L_{\sigma}=(-\Delta)^{\sigma/2},\,h>0,
\end{equation}
with $m>1$. The precise result is stated in Theorem \ref{counter.ell}.  The reduction to power-like nonlinearity simplifies the calculations and is the most important case in the applications.  The
solutions satisfy $ u\rightarrow0 \text{ as } |x|\rightarrow\infty.$
Comparing this equation to \eqref{eq.1}, we notice that here we denote by $u=B(v)=v^{1/m}$ the unknown function. Equation \eqref{eqell.h} is posed in $\ren$ with nonnegative and integrable data $f(x)$. We want to describe the asymptotic behaviour of the solution as $|x|\to\infty$, more precisely its rate of
decay. We will focus on the dependence of the behaviour on the constant $h>0$. This constant  is important since it represents the time increment when discretizing
the evolution problem. We stress the dependence by often denoting the solution as $u(x;h)$. Our main result says that roughly speaking
\begin{equation}
u(x;h) \sim \,h\,|x|^{-(N+\sigma)}
\end{equation}
when $|x|$ is large and $h$ small. We assume $m\ge 1$.
Note that the parameter $h$ can be changed, or fixed to 1, by using the scaling
\begin{equation}
\widetilde u(x)= a\,u(bx), \qquad \widetilde f(x)= a\,f(bx)\,.
\end{equation}
If $a^{m-1}b^{\sigma}\widetilde h=h$, then $\widetilde u$ satisfies: $\widetilde h L_{\sigma}\, (\widetilde u^m)+\widetilde u =\widetilde f$. In
other
words, $u(x;\widetilde h)=a u(bx;h)$. This will be of great use in deriving the negative implication for the concentration analysis in
Section~\ref{sec.neg}.

%%%%%%%%%%%%%%%%%%%%%%%%%%%%%%%%%%%%%%%%%%%%%%%%%%%%%%%%%%%%%%%%%
\subsection{Elliptic positivity estimate via subsolutions}

The first step in our asymptotic positivity analysis of solutions of \eqref{eqell.h} is to ensure that solutions with positive data remain
positive in
some
region. We only need a special case that we establish next.

\begin{lemma} \label{lemma.uniflower} Let $u(x;h)$ be the solution of \eqref{eqell.h} with RHS  $f(x)\ge0$ such that $f(x)\ge 1$ for $|x|\le 1$. We
assume that $m>0$. Then, for every $R<1$ there  are constants $A_1, h_1>0$ (depending on $R$) such that
\begin{equation}
u(x;h)\ge A_1 \quad \mbox{ for } \quad |x|\le R, \ 0<h<h_1.
\end{equation}
\end{lemma}

 {\sl Proof.}  $\bullet$  First, we construct a subsolution for a related problem that has an explicit form and compact support. Let $g(x)$ be  the
 explicit  function,
$$
g(x)=\frac12 (1-r^2)_+^{\sigma/2}, \quad r=|x|\ge 0\,.
$$
Getoor \cite{Getoor}, Theorem 5.2, proves that $L_{\sigma} g(x)=c_0>0$ on the ball of radius 1 where $g$ is positive, while $L_{\sigma/2}g<0$ for
$r>1$,
with
an
explicit formula that goes to minus infinity as $r\to 1$ and behaves as  $\sim r^{-(N+\sigma)}$ when $r\to\infty$. Next, we consider the following
combination
$$
f_1(x):=h L_{\sigma}g + g^{1/m}.
$$
This can be seen as follows: the solution $u$ of equation \eqref{eqell.h} corresponding to RHS  $f_1$ is  $u_1=g^{1/m}$.

Let us now try to estimate $f_1$: for $r>1$ we have $f_1<0$. For $r\le 1$ we have $f_1= c_o h+ g^{1/m}>0$, besides $f_1\le (1/2)^{1/m}+hc_0<1$ if
$h<
(1-2^{-m})/c_0$.  Under these restrictions on $h$,  $u_1=g^{1/m}$, the solution for RHS $f_1$, serves as a subsolution for the RHS
$f(x)=\chi_{1}(0)$,
the
characteristic function of the ball of radius $1$. This means that the solution $u$ corresponding to such $f$ is equal or larger than $u_1=g^{1/m}$.
Since
$u_1(x)$ is uniformly positive in the ball of radius $1/2$, $u(x)$ is uniformly positive in the ball of radius $1/2$ when  $0<h<h_1$.

$\bullet$ By means of  scalings to put the dimensions in $x$, $u$ and $h$ as in the statement. \qed

We now proceed  with the asymptotic estimate from below.

\begin{theorem} \label{lower.ell.est} Let $u(x;h)$ be the solution of \eqref{eqell.h} with RHS  $f(x)\ge0$ such that $f(x)\ge 1$ in the ball
$B_1(0)$.
We assume that $m\ge 1$. Then there are constants $ C_-,R_1, h_1 >0$  such that
\begin{equation}
u(x;h)\ge C_-\,h\,|x|^{-(N+\sigma)}
\end{equation}
if \ $|x|\ge R_1$ and $0<h<h_1$.
\end{theorem}

\noindent {\sl Proof.}  We will use the standard comparison theorem to reduce the case where $f$ is a smooth version of the
characteristic
function of
the ball $B_1(0)$. Then known theory says that $u\le 1$ everywhere and is continuous, radially symmetric and decreasing in $r=|x|$. In fact, a
bootstrap
argument
shows that $u\in C^{\infty}$. Since $\|u\|_1\le \|f\|_1=\omega_N$ we also have a first decay for $u$ near infinity of the form
$$
u(r)\le C\,r^{-N}.
$$
Of course, this first estimate is not sharp, in view of our next results.

\noindent $\bullet$ We want to construct a subsolution of the form
\begin{equation}
U^m(x;h)=G(|x|)+ h^m\,F^m(|x|)\,,
\end{equation}
which will be valid for $0<h<h_1$. Here the functions $G, F\ge 0$ and the constant $h_1>0$ have to chosen carefully, as explained below.

We take $G(r)=0$ for $r=|x|\ge 1/2$ so that $U(x;h)=hF(x)$ there. If $G$ is also smooth we have $L_{\sigma}\,G$ bounded and we can also choose $G$ so
that
$$
L_{\sigma}G\le -C_1r^{-(N+\sigma)} \quad \mbox{for} \quad r>1/2.
$$
We may choose as $G$ a smoothed version of the previous Getoor function, using  convolution.

We also need $F$ positive, smooth  and $F(r)\sim  C_2r^{-(N+\sigma)}$ as $r\to\infty$ to get the desired conclusion after the comparison argument:
$u(x;h)\ge
U(x;h)\ge C\,h\,r^{-(N+\sigma)}$ (if  $r$ is large and  $h\sim 0$, see below)

To check the property of subsolution we proceed as follows. We have
$$
L_{\sigma}U^m=L_{\sigma} G(x)+ h^m L_{\sigma}F^m(x)
$$
As we have pointed out, our choice of $G$ leads to the above  estimate for $L_{\sigma}G$ with negative sign. We also have $F\le C_2r^{-(N+\sigma)}$
for $r>1/2$, by Lemma 2.1 in
\cite{BV2012} we have  that since $F^m=O(r^{-(N+\sigma)m})$ and $(N+\sigma)m>N$, we can choose $F$ so  that  $ |L_{\sigma}F^m|\leq
C_{3}r^{-(N+\sigma)}$
for some
positive
constant $C_3$ and $r>1/2$. Then we will have
$$
U+ h\,L_{\sigma}U^m\le  h\left(F + L_{\sigma}G+h^mL_{\sigma}F^m\right) \le h( C_2r^{-(N+1)} -C_1r^{-(N+1)}+ h^m L_{\sigma}F^m),
$$
which will be negative for all $r>1/2$ if
$$
C_2+h^m C_3 < C_1\,.
$$
In order to make sure that such constants can be obtained, we fix first $G$ and this determines $C_1$. We then use a tentative $F_0$ for the
function
$F$ and
multiply it by a small constant so that $F$ and $L_{1/2}F^m$ are smaller than $C_1/2$. Indeed, we can take $C_{2}<C_{1}/2$ and
$h<(C_{1}/2C_{3})^{1/m}=:h_1$. Finally, we may take $h_1=1$.

\medskip

\noindent $\bullet$ The next step is to use the viscosity method to compare $u$ and $U$ in the  $Q=\{|x|\ge 1/2\}$, and this will prove that $U(x;h)\le u(x;h)$ in $Q$ if $h<h_1$.

The following inequality establishes a suitable comparison of the boundary conditions at $|x|=1/2$:
$$
U(x;h)=hF(1/2)< A_1\le u(x;h).
$$
Here we use Lemma \eqref{lemma.uniflower} with the choice $R=1/2$, which gives $u(x;h)\geq A_{1}$ for some constant $A_{1}$, $|x|\leq1/2$ and
$h$
sufficiently
small. Now
$$
hF(1/2)\leq 2^{N+1}\,h_{1}C_{2}<A_{1}
$$ up to choose $C_{2}$ properly and $h$ under a further bound.
Once this is justified, we argue at the first point where $u$ and $U$ touch. Actually, we must use and approximation $u_\ve$ instead of $u$. \qed

\medskip

\noindent  {\bf Remark.} The only restriction on $m$ is $m(N+\sigma)>N$, which means $m>m_1=N/(N+\sigma)$.

%%%%%%%%%%%%%%%%%%%%%%%%%%%%%%%%%%%%%%%%%%%%%%%%%%%%%%%%%%%%%%%%%
%%%%%%%%%%%%%%%%%%%%%%%%%%%%%%%%%%%%%%%%%%%%%%%%%%%%%%%%%%%%%%%%%
\subsection{Upper bound estimate}

\begin{theorem}\label{up.ell.est} Let $u(x;h)$ be the solution of \eqref{eqell.h} with RHS  $f$ such that $0\le f(x)\le 1$ in the ball
$B_1(0)$
and
$f(x)=0$
for $|x|>1$. We assume that $m\ge 1$. Then there are constants $ C_+,R_2, h_2 >0$ such that
\begin{equation}
u(x;h)\le C_+\,h\,|x|^{-(N+\sigma)}
\end{equation}
if \ $|x|\ge R_2$ and $0<h<h_2$.
\end{theorem}

\noindent {\sl Proof.}  We will construct a super-solution of the form
\begin{equation}
U^m(x;h)= G(x) + b^mh^m\,F(x)^m\,,
\end{equation}
where $F\ge 0$ is  chosen as before and $b>0$. We will use the fact that $F(r)\sim  C_1r^{-(N+\sigma)}$ as $r\to\infty$. It follows that there is a
large
constant
$k>0$ such
that
$$
kF+L_{\sigma}\,F^m\ge 0 \quad \mbox{everywhere in } \ \ren.
$$
Next we choose $G\ge 0$ compactly supported in a ball of radius $R_1>1$, and such that $L_{\sigma}G=c_0>0$ on the support. As $r\to\infty$,  we get
the
usual
$L_{\sigma
}G\sim -C\,r^{-(N+\sigma)}$. We also need $G(1)> 1$. Note that for $G=0$ we have $U=bhF$.
In any case $U$ is nonnegative, $U\ge 0$. We also have
$$
L_{\sigma}U^m=L_{\sigma}G+ b^mh^m L_{\sigma}F^m(x)
$$
We perform an analysis by regions. Thus, when $G=0$ we  have
 $$
U+ hL_{\sigma}U^m = h(b F + b^mh^mL_{\sigma}F^m + L_{\sigma}G)\ge 0
$$
The final inequality is obtained as follows: we first put $b>b_0$ so that $$(b/2)F  + L_{\sigma}G\ge 0$$ (recall that
$L_{\sigma}G=O(r^{-(N+\sigma)})$). Then
we
put
$b>2k(bh)^m$ to have $$(b/2)F + (bh)^mL_{\sigma}F^m\ge 0,$$ i.e. if $b^{m-1}h^mk<1/2$; this imposes an upper bound  on $h$.

On the other hand, where $G>0$ we have
$$
U+ h L_{\sigma}U^m \ge h(c_0+(bh)^mL_{\sigma}F^m)\ge 0
$$
if $C_3(bh)^m\le c_0$ (we use the fact that $L_{\sigma}F^m$ is bounded). Both conditions are fulfilled if $0<h<h_2$.

\medskip

\noindent $\bullet$ Now the viscosity method works in the region $Q=\{|x|\ge 1\}$, and this will prove that $U(x,h)\ge
u(x,h)$ in
$Q$.
Indeed, the boundary condition at $r=1$ is
$$
U(1)\ge G(1)\ge 1\geq u(x;h).
$$
 by the maximum principle. This ends the proof. \qed

\noindent  {\bf Remark.} The only restriction on $m$ is $m(N+\sigma)>N$, which means $m>m_1=N/(N+\sigma)$.

%%%%%%%%%%%%%%%%%%%%%%%%%%%%%%%%%%%%%%%%%%%%%%%%%%%%%%%%%%%%%%%%%%%%%%%%%%%%%%%%%%%%%%%%

\subsection{Scaled data. Negative concentration result}\label{sec.neg}

 We have done the argument for a solution with data of height 1 supported in the ball of radius $R=1$. If we want to change
the
radius to
$R\ne 1$ and the height to $A$ we can use the scaling
\begin{equation}\label{scal.frm}
\widetilde u(x;h)= A u(x/R; A^{m-1}R^{-\sigma} h)
\end{equation}
Using this formula and the result of Theorem \ref{lower.ell.est} applied to $u$, we see that the
comparison result holds in an $h$-interval of the form
$$
0<h<h_1(R)=h_1\,R^{\sigma}A^{-(m-1)},
$$
and the new result is
\begin{equation}
\widetilde u(x; h)\ge C_-\frac{hA^{m}R^N}{|x|^{N+\sigma}},
\end{equation}
valid for large $x$ and $h$ suitably small. We are interested in conservation of mass, i.\,e., $A=R^{-N}$. In that case, denoting the new
solution by
$u_R(x;h)$
we have
\begin{equation}
u_R(x; h)\ge \frac{C_-}{R^{N(m-1)}}\frac{h}{|x|^{N+\sigma}}.
\end{equation}

\noindent $\bullet$ In the same way, the scaling formula \eqref{scal.frm} applies in combination with the result of Theorem \ref{up.ell.est} in
the
$h$-interval:
$0<h<h_2(R)=h_2\,RA^{-(m-1)},$ and the new result is
\begin{equation}
u(x; h)\le C_+\frac{hA^mR^N}{|x|^{N+\sigma}}
\end{equation}
Under  conservation of mass, $A=R^{-N}$, denoting the solution by $u_R(x;h)$ we have
\begin{equation}
u_R(x; h)\le \frac{C_+}{R^{N(m-1)}}\frac{h}{|x|^{N+\sigma}}.\label{C_+}
\end{equation}

 We are ready to arrive at a contradiction in the comparison of concentrations.

 \begin{theorem} \label{counter.ell} If $m>1$ there exist two nonnegative, compactly supported, bounded, radially symmetric and rearranged functions,
 $f$ and $f_R$,
 such
 that $f_R
 \prec f$,  and nevertheless the corresponding solutions $u(x;h)$ and $u_R(x;h)$ do not obey the same relation.
 \end{theorem}

 \noindent {\sl Proof.} Let us choose $f=\chi_{1}$, then rescaled function $ f_R (x)=R^{-N}f(x/R)=R^{-N}\chi_{R}$ is compactly supported in the
 ball
 $B_{R}$,
 has
 height $R^{-N}$ and it less concentrated than $f$ if $R>1$. Let us consider the solutions $u$ and $u_R$ that they produce, with the same
 coefficient
 $h$. If we
 apply Theorem \ref{lower.ell.est} and inequality \eqref{C_+} to $u$ and $u_{R}$ respectively, we have that $u_R(x)<u(x)$ if
 $C_+<C_-R^{N(m-1)}$, on the condition that
 $x$ is
 large enough, and $h$ is small enough:
$$
|x|\ge R_-, \quad |x|\ge R_+R; \qquad h<h_1, \quad h< h_2R^{N(m-1)+\sigma}.
$$
Since $u_R$ and $u$ have the same mass (because they are solutions to equation \eqref{eqell.h} with the same $h$ and data having the same mass, see
the remark below), this means that
$u_R$ cannot be less concentrated than $u$ for such small values of $h$.

Once we have the contradiction for the equation with some $h$ we may put $h=1$ by scaling. \qed

\noindent  {\bf Remark.} Here we see that the contradiction is obtained only for $m>1$. For $m\le 1$,  $C_+$ will always be larger than
$C_-R^{N(m-1)}$, and
there is contradiction, just as predicted by the theory, cf. Theorem \ref{thm.ell.convex2}.

\noindent  {\bf Remark.} In order to prove the conservation of the mass for the nonlinearity $A(u)=u^{m}$ with $m>1$, we can argue as in Proposition
\ref{prop.7}.
Indeed, suppose that  $v$ be the weak solution of $ (-\Delta)^{\sigma/2}v+  B(v)=f$ with  $f\in L^1(\ren)$ nonnegative and let $u= B(v)$, with
$B:\R_{+}\rightarrow\R_{+}$ be a concave function, strictly increasing, such that $B(0)=0$. Suppose first that $f\in L^{\infty}(\ren)$ and $|f|\leq
K$.
By Theorem \ref{th.exist}, we have that $u=B(v)\leq K$. By the convexity of $A=B^{-1}$, we have that the function $t\in\R_{+}\rightarrow
A(t)/t\in\R_{+}$ is increasing, then
$$
\frac{v}{u}=\frac{A(B(v))}{B(v)}\leq \frac{A(K)}{K}.
$$
Now if $\zeta$ is the usual cutoff function and $\zeta_R(x)=\zeta(Rx)$ is its rescaled version, we still find equation \eqref{conservation}. We also
have
$$
\left|\int_{\ren} v\,(-\Delta)^{\sigma/2}\zeta_R\,dx\right| \le \frac{A(K)}{K}\int_{\ren } |u(x)(-\Delta)^{\sigma/2}\zeta_R|\,dx \le
\frac{c}{R^{\sigma}} \int_{\ren} |u(x)|\,dx
$$
which in the limit $R\to \infty$ tends to zero. Then by \eqref{conservation} we conclude
$$
\int_{\ren} u(x)\,dx=\int_{\ren} f(x)\,dx.
$$
If $f$ is in $L^{1}(\ren)$ we proceed by approximation.

\noindent  {\bf Last Remark.} We want to point out the comparison performed in this section looks too contrived. Actually, this is partly due to
the fact that it is the translation into the elliptic framework of the more natural parabolic counterexample, constructed in Section
\ref{sec.neg.par}.

%%%%%%%%%%%%%%%%%%%%%%%%%%%%%%%%%%%%%%%%%%%%%%%%%%%%%%%%%%%%%%%%%%%
%%%%%%%%%%%%%%%%%%%%%%%%%%%%%%%%%%%%%%%%%%%%%%%%%%%%%%%%%%%%%%%%%%%

\section{Symmetrization for the parabolic problem}\label{sec.par}
\setcounter{equation}{0}

For simplicity of exposition, we start with the case $f=0$.
We briefly remind the concept of \emph{weak solution} to the Cauchy problem associated to a fractional parabolic equation :
\begin{equation} \label{eqcauchy}
\left\{
\begin{array}
[c]{lll}%
u_t+(-\Delta)^{\sigma/2}A(u)=0  &  & x\in\R^{N}\,,t>0%
\\[6pt]
u(x,0)=u_{0}(x) &  & x\in\R^{N}.
\end{array}
\right. %
\end{equation}
Here, $u_{0}$ is an  integrable function on $\R^{N}$, nonnegative in our applications), the nonlinearity $A(u)$ is a nonnegative concave function with $A(0)=0$ and $A'(u)>0$ for all
$u>0$ (extended antisymmetrically in the general two-signed theory). Set $B:=A^{-1}$. As in the elliptic case we rewrite problem \eqref{eqcauchy} as the following quasi-stationary problem
\begin{equation} \label{eqcauchy1}
\left\{
\begin{array}
[c]{lll}%
-\operatorname{div}_{x,y}\left(  y^{1-\sigma}\nabla w\right)  =0  &  & (x,y)\in\R^{N}\times(0,\infty)\,,t>0
\\[6pt]
\dfrac{1}{\kappa_{\sigma}}\displaystyle{\lim_{y\rightarrow0^+}y^{1-\sigma}\,\dfrac{\partial w}{\partial y}}-\dfrac{\partial B(w)}{\partial t}=0 &  &
x\in\R^{N},\,y=0,t>0,\\
[8pt]
 w(x,0,0)=A(u_{0}(x))&  & x\in\R^{N}.
\end{array}
\right. %
\end{equation}
Then we have the following definition

\begin{definition}\label{def.weak.par}
We say that $w$ is a weak  energy solution to problem \eqref{eqcauchy1}
 if $w\in L^{2}_{loc}((0,\infty);X^{\sigma/2}(\mathcal{C}_{\ren}))$, the function $u(x,t):=(B(w(x,0,t))$ is in the space $C([0,\infty);L^{1}(\ren))$
 and the following identity holds
\[
\int_{0}^{\infty}\int_{\R^{N}}u\frac{\partial
\varphi}{\partial
t}\,dx\,dt-\frac{1}{\kappa_{\sigma}}\int_{0}^{\infty}\int_{\mathcal{C}_{\ren}}y^{1-\sigma}\,\nabla_{x,y}\,w\cdot\nabla_{x,y}\,\varphi\,dx\,dy\,dt=0
\]
for all test functions $\varphi\in C_{0}^{1}(\overline{\R_{+}^{N+1}}\times[0,\infty))$. Finally, the initial data are taken in the sense that
\[
\lim_{t\rightarrow0}u(\cdot,t)=u_{0}(x)\quad \in\,L^{1}(\R^{N}).
\]
\end{definition}

If $w$ is a solution to \eqref{eqcauchy1}, we sill say that $u(x,t):=(B(w(x,0,t))$ is a \emph{weak solution} to the Cauchy problem \eqref{eqcauchy}. We refer to \cite{pqrv}, \cite{pqrv2} for questions related to existence and uniqueness of weak solutions to problem \eqref{eqcauchy}.

 The theorems we are going to prove, Theorems \ref{thm.par.convex} and \ref{thm.par.convex.f}, will come from the combination of two ingredients:  the  existence of a \emph{mild solution} to problem \eqref{eqcauchy1}  that  is  reduced to solving some elliptic problems by applying the
Crandall-Liggett theory for $m$-accretive operators, and the comparison theorems \ref{genconvex}, \ref{thm.ell.convex2}, already proved for elliptic problems. Therefore, we will devote a subsection to review this material for the reader's convenience.

%%%%%%%%%%%%%%%%%%%%%%%%%%%%%%%%%%%%%%%%%%%%%%%%%%%%%%%%%%%%%%%%%%%%%%%

\subsection{Abstract evolution equations and accretive operators. The semigroup approach}
\label{Appendix}
\setcounter{equation}{0}

Let $X$ be a Banach space and $\mathcal{A}:D(\mathcal{A})\subset X\rightarrow X$ a nonlinear operator defined on a suitable subset of $X$. Let us consider the problem
\begin{equation}\label{eqcauchyabstract.3}
\left\{
\begin{array}
[c]{lll}%
u^{\prime}(t)+\mathcal{A}(u)=f,  &  & t>0,%
\\[4pt]
u(0)=u_{0}\,, &  &
\end{array}
\right.
\end{equation}
where $u_{0}\in X$ and $f\in L^{1}(I;X)$ for some interval $I$ of the real axis.  For a wide class of operators, in particular the ones considered in this paper, a very efficient way to approach such problem is to use an implicit time discretization scheme that we describe next. Suppose to be specific that
$I=[0,T]$ (but this can be replaced by any interval $[a,b]$ and the procedure is similar). The method
consists in taking first a partition of the interval, say, $t_k=kh$ for $k=0,1,\ldots n$ and $h=T/n$, and then solving the system of difference relations
\begin{equation}
\frac{u_{h,k}-u_{h,k-1}}{h}+\mathcal{A}(u_{h,k})=f_{k}^{(h)}\label{discrellprob}
\end{equation}
for $k=0,1,\ldots n$, where we pose $u_{h,0}=u_{0}$. The data set $\left\{f_{k}^{(h)}:k=1,\ldots,n\right\}$ is supposed to be a suitable discretization
of
the source term $f$, corresponding to the time discretization we choose. This process is called \emph{implicit time discretization scheme} (ITD for
short)
of the equation
$u^{\prime}(t)+\mathcal{A}(u)=f$. It can be rephrased in the form
\[
u_{h,k}=J_{h}(u_{h,k-1}+hf_{k}^{(h)})
\]
where the operator
\[
J_{\lambda}=(I+\lambda \mathcal{A})^{-1},\,\lambda>0
\]
is called the \emph{resolvent operator}, being $I$ the identity operator.  Therefore, the application of the method needs the operator $\mathcal{A}$
to
have a well-defined family of resolvents with good properties.  When the ITD is solved, we construct a \emph{discrete approximate solution}
$\left\{u_{h,k}\right\}_{k}$. By piecing together the values  $u_{h,k}$ we form a piecewise constant function, $u_h(t)$,
typically defined through
\begin{equation}
u_{h}(t)=u_{h,k}\quad\text{if }t\in[(k-1)h,kh]\label{interpol}
\end{equation}
(or some other interpolation rule, like linear interpolation). Then the main question consists in verifying if such function $u_{h}$ converges
somehow
as
$h\rightarrow0$ to a solution $u$ (which we hope to be a classical, strong, weak, or other type of solution) to problem \eqref{eqcauchyabstract.3}.
To
this
regard, we first choose a suitable discretization $\left\{f_{k}^{(h)}\right\}$  in time of
the source term $f$, such that the piecewise constant interpolation of this sequence produces a function $f^{(h)}(t)$ (defined by means of
\eqref{interpol}) verifies the property
\[
\|f^{(h)}-f\|_{L^{1}(0,T;X)}\rightarrow0\quad\text{as }h\rightarrow0.
\]
By means of these discrete approximate solutions we introduce the following notion of \emph{mild solution\,}:

\begin{definition}
We say that $u\in C((0,T);X)$ is a mild solution to \eqref{eqcauchyabstract.3} if it is obtained as uniform limit of the approximate solutions $u_{h}$, as $h\rightarrow0$. The initial data are taken in the sense that $u(t)$ is continuous in $t=0$ and $u(t)\rightarrow u_{0}$ as $t\rightarrow0$.
Besides, we say that $u\in C((0,\infty);X)$ is a mild solution to \eqref{eqcauchyabstract.3} in $[0,\infty)$ if $u$ is a mild solution to the
same problem in any compact subinterval $I\subset [0,\infty)$.
\end{definition}

In order to state a positive existence result, we  need to restrict the class of operators according to the following definitions.

\begin{definition}\label{AcRank}Let $\mathcal{A}:D(\mathcal{A})\subset X\rightarrow X$ be a nonlinear,  possibly unbounded operator. Let
$R_{\lambda}(\mathcal{A})$ be the range of $I+\lambda \mathcal{A}$, a subset of $X$.

\noindent {\rm (i)}  The operator $\mathcal{A}$ is said accretive if for all $\lambda>0$ the map $I+\lambda \mathcal{A}$ is one-to-one onto
$R_{\lambda}(\mathcal{A})\subset X$, and the resolvent operator $J_{\lambda}:R_{\lambda}(\mathcal{A})\rightarrow X$ is a (non-strict) contraction in
the
$X$-norm (i.\,e., a Lipschitz map with Lipschitz norm 1).

 \noindent {\rm  (ii)} We say that $\mathcal{A}$ satisfies the rank condition if \ $R_{\lambda}(\mathcal{A})\supset\overline{D(\mathcal{A})}$ for all
 $\lambda>0$. In particular, the rank condition is satisfied if \ $R_{\lambda}(\mathcal{A})=X$ for all $\lambda>0$; in this case, if $\mathcal{A}$ is accretive, we say that $\mathcal{A}$ is
 $m$-accretive.
\end{definition}

We are now ready to state the desired semigroup generation result, that generalizes the classical result of Hille-Yosida (valid in Hilbert spaces and
for
linear $\mathcal{A}$) and the variant by Lumer and Phillips (valid in Banach spaces, still for linear $\mathcal{A}$), and provides the existence and
uniqueness of mild solutions for problems of the type \eqref{eqcauchyabstract.3} in the case $f\equiv0$:

\begin{theorem}[Crandall-Liggett]\label{CrLigg} Suppose that $\mathcal{A}$ is an accretive operator satisfying the rank condition. Then for all data
$u_{0}\in \overline{D(\mathcal{A})}$ the limit
\begin{equation}
S_{t}(\mathcal{A})u_{0}=\lim_{n\rightarrow\infty}(J_{t/n}(\mathcal{A}))^{n}u_{0}.\label{CrLig}
\end{equation}
exists uniformly with respect to $t$, on compact subset of $[0,\infty)$, and  $u(t)=S_{t}(\mathcal{A})u_{0}\in C([0,\infty): X)$. Moreover, the family of operators $\left\{S_{t}(\mathcal{A})\right\}_{t>0}$ is a strongly continuous semigroup of contractions on $\overline{D(\mathcal{A})}\subset X$.
\end{theorem}

Using a popular notation in the linear framework, we could write $S_{t}(\mathcal{A})u_{0}=e^{-t\mathcal{A}}u_{0}$, and because of this analogy formula
\eqref{CrLig} is called the \emph{Crandall-Liggett exponential formula for the nonlinear semigroup generated by} $-\mathcal{A}$. The problem with this very general and useful result is that the $X$-valued function $u(t)=S_t(\mathcal{A})u_0$ solves the equation only in a mild sense, that is not necessarily
a strong solution or a weak solution. Though it is known that strong solutions are automatically mild, the correspondence between mild and weak solutions is not always clear. For the FPME this issue has been discussed in detail in \cite{pqrv, pqrv2}.

In addition, the Crandall-Liggett Theorem result can be extended when we consider nontrivial source term $f$, according to the following result

 \begin{theorem}\label{existmildsol}
 Suppose that $\mathcal{A}$ is $m$-accretive. If  $f\in L^{1}(0,\infty;X)$ and $u_{0}\in\overline{D(\mathcal{A})}$. Then the abstract  problem
 \eqref{eqcauchyabstract.3} has a unique mild solution $u$, obtained as limit of the discrete approximate solution $u_{h}$ by ITD scheme described
 above,
 as $h\rightarrow0$:
 \[
 u(t):=\lim_{h\rightarrow0}u_{h}(t)\,,
 \]
 and the limit is uniform in compact subsets of $[0,\infty)$. Moreover, $u\in C([0,\infty);X)$ and for any couple of solutions $u_{1}$, $u_{2}$
 corresponding to source terms $f_1$, $f_2$ we have
 \[
 \|u_{1}(t)-u_{2}(t)\|_{X}\leq\|u_{1}(s)-u_{2}(s)\|_{X}+\int_{s}^{t}\|f_{1}(\tau)-f_{2}(\tau)\|_{X}d\tau
  \]
 for all $0\leq s<t$.
 \end{theorem}

There is a wide literature on these topics, starting with the seminal paper by Crandall and Liggett \cite{CL71}, see also \cite{Cr86} and the general
reference \cite{Barbu}.  These notes are based on Chapter 10 of the book \cite{vazquezPME}, cf. the references therein.  The last formula we have
mentioned
introduces the correct concept of uniqueness for the constructed class of solutions. Characterizing the  uniqueness of different concepts of solution
is a
difficult topic already discussed by B\'enilan in his thesis \cite{BeTh}.

\subsection{Parabolic Symmetrization}

In order to apply this theory we have to check that the operator associated to our evolution problem $\mathcal{\mathcal{A}}$ is $m$-accretive or that
it
is
accretive and the rank condition holds, in the sense of definition \ref{AcRank}. A main question in this approach to nonlinear evolution is the
corrected
identification of the operator. This has been done in \cite{pqrv}  as follows.

If $u_{0}\in L^{1}(\ren)\cap L^{\infty}(\ren)$, we introduce the nonlinear operator
$\mathcal{A}_0: D(\mathcal{A}_0)\subset L^{1}(\ren)\rightarrow L^{1}(\ren)$, defined by
$$
\mathcal{A}_0(u):=(-\Delta)^{\sigma/2}A(u)\,,
$$
with domain
$$
D(\mathcal{A}_0 ):=\left\{v\in L^1(\ren)\cap L^{\infty}(\ren): \mathcal{A}_0(v)\in L^1(\ren)\cap L^{\infty}(\ren) \right\}.
$$
Returning to the results of Subsection \ref{sect.ell}.1, we see that the
contractive property \eqref{contraction} implies that this operator is accretive in the space $X=L^1(\ren)$. On the other hand, Theorem
\ref{th.exist}
and
its extension on $\ren$ gives the rank condition in $L^1\cap L^\infty$. By closing this operator with respect to the norm $\|.\|_1$ we find an
operator
$\mathcal{A}$ that is $m$-accretive in $L^1(\ren)$.

Therefore, we can use Theorem \ref{existmildsol} that implies that there is a unique
mild solution to \eqref{eqcauchy}, obtained as a limit of discrete approximate solutions by the ITD scheme. In the case $A(u)=u^m$ the extra
regularity of these solutions is discussed in detail in the papers \cite{pqrv, pqrv2}. For general $A$ see \cite{pqrv4}.

We can now use this method to prove a symmetrization result for Fractional Fast Diffusion Equations, including in particular the  well-known linear
fractional heat equation,
$$u_t+(-\Delta)^{\sigma/2}u=0.$$

\begin{theorem}\label{thm.par.convex} Let $u$ be the mild  nonnegative solution of  the FPME \eqref{nolin.parab} with $0<\sigma<2$, posed in $\Omega=\ren$,
with initial data $u_0\in L^1(\ren)\ge 0$ and nonlinearity $A(u)$ given by a concave function with $A(0)=0$ and $A'(u)>0$ for all $u>0$. Let $v$ be
the solution of the corresponding symmetrized problem
\begin{equation} \label{eqcauchysymm}
\left\{
\begin{array}
[c]{lll}%
v_t+(-\Delta)^{\sigma/2}A(v)=0  &  & x\in\R^{N}\,, \ t>0,%
\\[6pt]
v(x,0)=u_{0}^{\#}(x) &  & x\in\R^{N}\,.
\end{array}
\right. %
\end{equation}
Then we have for all $t>0$
\begin{equation}
u^\#(|x|,t)\prec v(|x|,t).\label{concentrationcomp}
\end{equation}
In particular, we have $\|u(\cdot,t)\|_p \le\|v(\cdot,t)\|_p$ for every $t>0$ and every $p\in [1,\infty]$.
\end{theorem}

\noindent{\sl Proof.}
According to what explained before, we use the implicit time discretization scheme. For each time $T>0$, we divide the time interval $[0,T]$ in $n$ subintervals
$(t_{k-1},t_{k}]$, where $t_{k}=kh$ and $h=T/n$. We construct then the function $u_{h}$ which is piecewise constant in each interval
$(t_{k-1},t_{k}]$,
by
\[
u_{h}(x,t)=
\left\{
\begin{array}
[c]{lll}%
u_{h,1}(x)  &  & if\,\,t\in[0,t_{1}]%
\\[6pt]
u_{h,2}(x)  &  & if\,\,t\in(t_{1},t_{2}]
\\
[6pt]
\cdots
\\
[6pt]
u_{h,n}(x)  &  & if\,\,t\in(t_{n-1},t_{n}]
\end{array}
\right. %
\]
where $u_{h,k}$ solves the equation
\begin{equation}
h(-\Delta)^{\sigma/2}A(u_{h,k})+u_{h,k}=u_{h,k-1}\label{eq.18}
\end{equation}
with the initial value $u_{h,0}=u_{0}$.
Similarly, concerning the symmetrized problem \eqref{eqcauchysymm}, we define the piecewise constant function $v_{h}$ by \[
v_{h}(x,t)=
\left\{
\begin{array}
[c]{lll}%
v_{h,1}(x)  &  & if\,t\in[0,t_{1}]%
\\
[6pt]
v_{h,2}(x)  &  & if\,t\in(t_{1},t_{2}]
\\
[6pt]
\cdots
\\
[6pt]
v_{h,n}(x)  &  & if\,t\in(t_{n-1},t_{n}]
\end{array}
\right. %
\]
where $v_{h,k}(x)$ solves the equation
\begin{equation}
h(-\Delta)^{\sigma/2}A(v_{h,k})+v_{h,k}=v_{h,k-1}\label{eq.20}
\end{equation}
with the initial value $v_{h,0}=u_{0}^{\#}$.
Our aim is now to compare the solution $u_{h,k}$ to \eqref{eq.18} with the solution \eqref{eq.20}. We proceed by induction. Using Theorem
\ref{genconvex}, we get
\[
A(u^{\#}_{h,1})\prec A(v_{h,1}).
\]
If we suppose by induction that $u^{\#}_{h,k-1}\prec v_{h,k-1}$ and call $\widetilde{u}_{h,k}$ the (radially decreasing) solution to the
equation
\begin{equation*}
h(-\Delta)^{\sigma/2}A(\widetilde{u}_{h,k})+\widetilde{u}_{h,k}=u^{\#}_{h,k-1},
\end{equation*}
Theorem \ref{genconvex} and Theorem \ref{thm.ell.convex2} imply
\begin{equation}\label{eq.21}
A(u^{\#}_{h,k})\prec A(\widetilde{u}_{h,k})\prec A(v_{h,k})\,,
\end{equation}
hence \eqref{eq.21} holds for all $k=1,\ldots,n$. Therefore, by the definition of $u_{h}$ and $v_{h}$, we find
\begin{equation}
A(u_{h}(\cdot,t)^{\#})\prec A(v_{h}(\cdot,t))\label{eq.22}
\end{equation}
for all times $t$. Using Lemma \ref{lemma1} with the choice $\Phi=F\circ B$, where $F\ge 0$ is convex and $F(0)=0$,  we obtain
\[
\int_{\ren} F(u_{h}^{\#}(x,t))dx\leq\int_{\ren} F(v_{h}(x,t))dx\,,
\]
which in turn yields
\begin{equation}
u_{h}^{\#}(\cdot,t)\prec v_{h}(\cdot, t).\label{conccomph}
\end{equation} Now Crandall-Liggett Theorem \eqref{existmildsol} implies
\[
u_{h}\rightarrow u,\quad v_{h}\rightarrow v\,\, \text{uniformly}.
\]
Then passing to the limit in \eqref{conccomph} we get the result.\nc\qed

\subsection{Symmetrization for the equation with a left-hand side}

We now consider the case $f\in L^1(Q)$, $Q=\ren\times(0,\infty)$ and $f\not\equiv 0$. In that case the
semigroup generation Theorem \ref{existmildsol} can still be applied to obtain the so-called {\sl unique mild solution } of the evolution problem \begin{equation}
\label{eqcauchy.f}
\left\{
\begin{array}
[c]{lll}%
u_t+(-\Delta)^{\sigma/2}A(u)=f  &  & x\in\R^{N}, \ t>0\,,%
\\[6pt]
u(x,0)=u_{0}(x) &  & x\in\R^{N}.
\end{array}
\right. %
\end{equation}
As explained in Subsection \ref{Appendix}, we need to perform a discretization of $f$ adapted to the time mesh $t_k=kh$ that
we have used above, let us call
it
$\{f_k^{(h)}\}$, so that the piecewise constant (or linear in time) interpolation of this  sequence produces a function $f^{(h)}(x,t)$ such that
$\|f-f^{(h)}\|_1\to 0$ as $h\to 0$. Then we use the previous implicit discretization scheme, now in the form
\begin{equation}\label{disc.ev.eqn}
\frac{1}{h}(u_{h,k}-u_{h,k-1})+(-\Delta)^{\sigma/2}A(u_{h,k})=f_k^{(h)}\,,
\end{equation}
to produce the semi-discrete function $\{u_{h}(x,t_k)=u_{h,k}(x): k=0,1,\cdots\}$, which after interpolation in time serves as $h$-approximation
to
the mild
solution $u(x,t)$. According to \eqref{disc.ev.eqn} we have to solve the elliptic problems
\begin{equation}\label{eq.18b}
h(-\Delta)^{\sigma/2}A(u_{h,k})+u_{h,k}=u_{h,k-1}+ h\,f_k^{(h)}\,,
\end{equation}
and we can use the theory developed in Section \ref{sect.ell}. Then we have the following result

\begin{theorem}\label{thm.par.convex.f} Let $u$ be the nonnegative mild solution of  the FPME \eqref{nolin.parab} with $0<\sigma<2$, posed in
$\Omega=\ren$, with
initial data $u_0\in L^1(\ren)$, $u_0\ge 0$, right-hand side $f\in L^1(Q)$, $f\ge 0$, and nonlinearity $A(u)$ given by a concave function with
$A(0)=0$
and
$A'(u)>0$ for all $u>0$. Let $v$ be the solution of the symmetrized problem
\begin{equation} \label{eqcauchysymm.f}
\left\{
\begin{array}
[c]{lll}%
v_t+(-\Delta)^{\sigma/2}A(v)=f^{\#}(|x|,t)  &  & x\in\R^{N}\,, \ t>0,%
\\[6pt]
v(x,0)=u_{0}^{\#}(x) &  & x\in\R^{N},
\end{array}
\right. %
\end{equation}
where $f^{\#}(|x|,t)$ means symmetrization of $f(x,t)$ w.r. to $x$ for a.e. time $t>0$. Then,
for all $t>0$ we have
\begin{equation}
u^\#(|x|,t)\prec v(|x|,t).
\end{equation}
In particular, we have $\|u(\cdot,t)\|_p \le\|v(\cdot,t)\|_p$ for every $t>0$ and every $p\in [1,\infty]$.
\end{theorem}

 The proof follows the lines of Theorem \ref{thm.par.convex}, so we leave the details to the reader.

\medskip

 \noindent {\bf Remark.} As an easy extension,  we can have also a result about comparison of concentrations for the radial solutions of two
 evolution
 problems, if we  assume  that the
 initial data satisfy the condition $u_{0,1}\prec u_{0,2}$ and the right-hand sides satisfy $f_1(\cdot,t)\prec f_{2}(\cdot,t)$ for almost all  $t>0$.
 The  conclusion is that $u_1(\cdot,t)\prec u_2(\cdot,t)$ for all $t>0$. Let us remind the reader that the result holds only if $A$ is linear or
 concave, as  assumed  above.

%%%%%%%%%%%%%%%%%%%%%%%%%%
\section{Negative result about concentration comparison for the Fractional PME}
\label{sec.neg.par}
\setcounter{equation}{0}

As in the elliptic case, it came to us as a surprise that the comparison result could not be proved for general nonlinearities $A$ without the
assumption of concavity. It turns out that for convex powers it does not hold. Here we will state and prove the negative result about concentration
comparison
for solutions of the Fractional Diffusion Equation in the range of exponents $m>1$, usually known as Slow Diffusion. We first argue in a formal way, since we give later the justification of some details.

\noindent $\bullet$ Let us consider the Fractional PME: \ $u_t+(-\Delta)^{\sigma/2}u^m=0$ in
$\ren$
and
consider first the Barenblatt solution that was studied in \cite{vazBaren}. Here we suppose that $$m>(N-\sigma)/N=:m_c.$$ Let us fix the mass to 1 for
simplicity. The Barenblatt solution has the form
$$
U_{1}(x,t)=t^{-\alpha}F_{m,1}(\xi), \qquad \xi=|x|t^{-\beta}
$$
with $\alpha=N\beta$ and $\beta=1/(N(m-1)+\sigma)$ and $F_{m,1}$ is  the Barenblatt profile of mass 1.  It is also known that as $\xi\to\infty$ we
have
\begin{equation}
F_{m,1}(\xi)\sim C\,\xi^{-(N+\sigma)}.\label{asymptobeh}
\end{equation}
This means that for large $x$ and $t\sim 0$ (so that $\xi\sim \infty$) we get
$$
U_{1}(x,t)\sim C\,t^{\lambda}|x|^{-(N+1)}, \qquad \lambda=\beta\sigma.
$$
This approximation holds uniformly for all $|x|\ge C$ large and all $0<t<\tau$ if $\tau$ is small enough, that is the error is higher order
small
$$
U_{1}(x,t)= C\,t^{\lambda}|x|^{-(N+\sigma)}(1+\varepsilon).
$$
Now, let us choose the initial data to associate to the equation. Let
$$
u_0(x)=\phi(|x|)\ge 0\,,
$$
where $\phi$ is a smooth and compactly supported function in the ball of radius one, having mass 1. Suppose also that $\phi$ is rearranged. Let
us call
$u$ the
solution to the equation with such choice of the data.  We have that $(-\Delta)^{\sigma/2}u_0^m$ is a bounded function that decreases at infinity
like
$C_1\,|x|^{-(N+\sigma)}$.
By virtue of the equation we have
$$
 u_t(x,0)=-(-\Delta)^{\sigma/2}u_0^m
$$
so that for small $t$ we have approximately
$$
 u(x,t)\sim - t\,(-\Delta)^{1/2}u_0^m
$$
and this behaves as $t\to\infty$ like $Ct|x|^{-(N+\sigma)}$. A rigorous proof of this behaviour
will be given in the next two subsections.

The conclusion is that the concentration comparison result is not true. The reason is that the exponent
$$
\lambda=
\sigma\beta=\frac{\sigma}{\sigma+N(m-1)}
$$
is less than 1 (precisely for $m>1$). Clearly, $U_0=\delta(x)$ is more concentrated than $u_0$, but $u(x,t)$ is larger than $U(x,t)$ for very large
$x$ if $t$ is quite small. This is incompatible with satisfying the concentration comparison and having the same mass.

In the graphics of Figure 1  we show the relative evolution of two solutions in time. Initially the concentrations are ordered, later they are not.
The parameter $m$ is 2.

\begin{figure}[!h]\label{figure_Rates}
\centering
\ifpdf
    % we are running pdflatex
    \includegraphics[width=0.8\textwidth, height=0.43\textwidth]{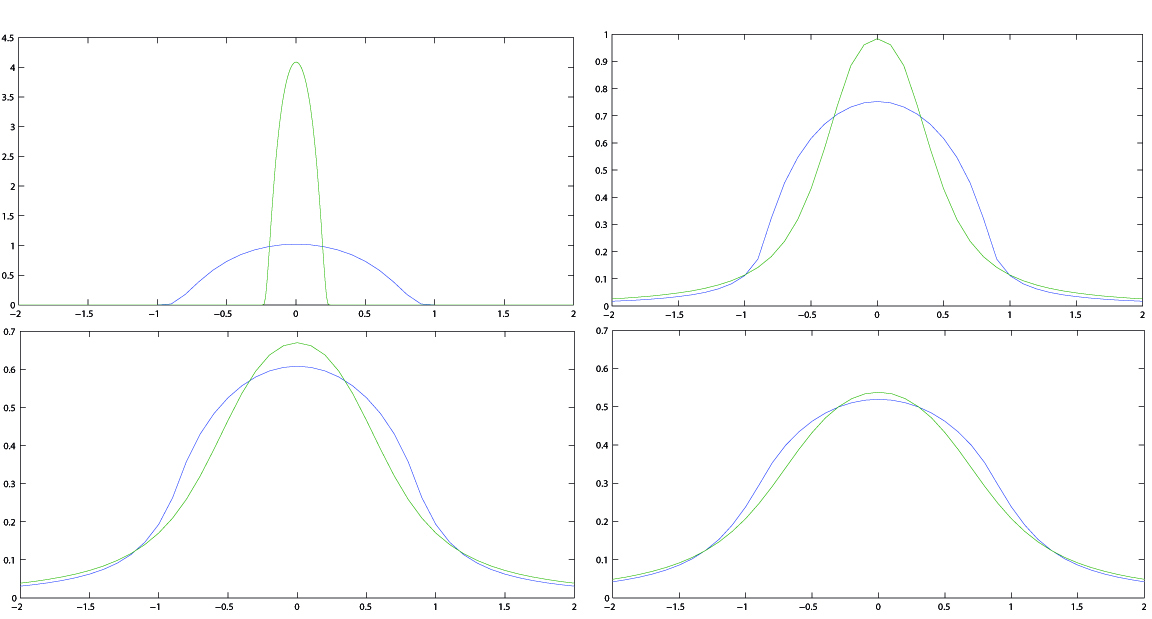}
\else
    % we are running latex
    \includegraphics[width=0.8\textwidth]{ColasJuntas2.eps}
\fi
\caption{Comparison of FPME evolution at four consecutive times.}
\end{figure}

%%%%%%%%%%%%%%%%%%%%%%%%%%%%%%%%%%%%%%%%%%%%%%%%%%%%%%%%%%%%%%%%%
\subsection{Supersolution. First tail estimate}

Let us now give a rigorous derivation of the tail behaviour. Fist, we have a preparatory step.

\begin{theorem}\label{upper.tail.par} Let $u(x,t)$ be a classical solution of the FPME with initial data $u_0(x)\ge0$ such that $u_0(x)\le 1$ in the
ball $B_1(0)$ and
$u_0(x)\le
\,|x|^{-(N+\sigma)}$ for $|x|>1$. Then there is a time $t_1>0$ such that
\begin{equation}\label{rate1.tail}
u(x,t)\le 2\, |x|^{-(N+\sigma)}
\end{equation}
if $|x|\ge R$ and $R$ is large enough and $0<t<t_1$.
\end{theorem}

\noindent {\sl Proof.} We consider the FPME for $m>1$ and initial data $u_0$  is  1 in the ball of radius 2.  Then for all times the solution
will be bounded by 1.

We want to construct a super-solution of the form
\begin{equation}
U(x,t)= (1+bt)\,F(|x|),
\end{equation}
where $F\ge 0$ has to chosen. I will need $F(r)\sim  C_1r^{-(N+\sigma)}$ as $r=|x|\to\infty$ to get the desired conclusion after the comparison
argument in
the
following way: $$u(x,t)\le U(x,t)\le 2F(r)\le 2C_1\,r^{-(N+\sigma)}$$ if  $r$ is large and  $t\le t_1/b \sim 0$.

To establish such comparison we first note  that \ $U_t=b F(x).$ Using the notation $$L_{\sigma}=(-\Delta)^{\sigma/2}$$, we also have
$$
L_{\sigma}U^m=(1+bt) L_{\sigma}F^m(x)
$$
As we have pointed out, $F(r)\sim  C_1r^{-(N+\sigma)}$ as $r\to\infty$ so that  $ L_{\sigma}F^m=O(r^{-(N+\sigma)})$ for $r>1$ (cf. Lemma 2.1 of
\cite{BV2012}). It follows
that there is a constant $k>0$ such that
$$
kF+L_{\sigma}\,F^m\ge 0, \quad \mbox{everywhere in } \ \ren.
$$
 Therefore, we will have
$$
U_t+ L_{\sigma}U^m =  bF + (1+bt)^mL_{\sigma}F^m\ge 0
$$
if $b>k(1+bt)$, i.e. if $b>k$ and $t<(b-k)/kb$, for instance for $b=2k$ and $t<1/b=1/2k$.

Under such assumptions, the viscosity method will work in the exterior region $Q=\{(x,t): |x|\ge 1, 0<t<t_1\}$, and this will prove that
$U(x,t)\le u(x,t)$ in $Q$ as desired.

We finally check the application of the viscosity method. Indeed, the boundary condition at $r=1$ is
$$
U(1,t)\ge F(1)\ge 1.
$$
so $U(x,t)\ge u(x,t)$ on the lateral boundary of $Q$ located at $r=1$. Same comparison is trivial for $t=0$. We only need to argue by contradiction
at
the
first point where the classical solution $u$ touches $U$ from below to conclude that $u(x,t)$ is strictly less than $U(x,t)$ in $Q$. The
contradiction
at
the point of contact is explained in \cite{BV2012}. The construction of classical solutions is done in \cite{pqrv4}.

\smallskip

\noindent {\bf Remarks.} (i) We have done the argument for $R=1$. If we want to change the radius to $R>1$ we may use the scaling of the equation.

(ii) Lower estimates that match the tail behaviour \eqref{rate1.tail} are derived and used in \cite{StanV2013}.

%%%%%%%%%%%%%%%%%%%%%%%%%%%%%%%%%%%%%%%%%%%%%%%%%%%%%%%%%%%%%%%%
\subsection{Supersolution. Sharp tail estimate}

\begin{theorem} Let $u(x,t)$ be a classical solution of the FPME with initial data $u_0(x)\ge0$ such that $u_0(x)\le 1$ in the ball $B_1(0)$ and
$u_0(x)=0$
for $|x|>1$. Then there is a time $t_1>0$ and constants $C^*$ and $R$ such that
\begin{equation}\label{rate2.tail}
u(x,t)\le C^* t\, |x|^{-(N+\sigma)}
\end{equation}
if $|x|\ge R$ and $R$ is large enough and $0<t<t_1$.
\end{theorem}

\noindent {\sl Proof.} We proof is a delicate variation of the preceding one. We still consider the FPME for $m>1$ and initial data $u_0$  is  1 in
the
ball of radius 2.  For all times the solution
will be bounded by 1. We consider a supersolution of the form
$$
U^m(x,t)= G(x) + b^mt^m\,F(x)^m,
$$
where $F\ge 0$ is  chosen as before. I take $F(r)\sim  C_1r^{-(N+\sigma)}$ as $r\to\infty$. Again, it follows that there is a constant $k>0$ such
that
$$
kF+L_{\sigma}\,F^m\ge 0, \quad \mbox{everywhere in } \ \ren.
$$
Next we choose $G\ge 0$ compactly supported and such that $L_{\sigma}G=c_0>0$ on the support. As $r\to\infty$,  we get the usual $L_{\sigma}G\sim
-cr^{-(d+2s)}$.

We get the formula
$$
U_t=b(G(x)+ b^mt^m\,F^m(x))^{(1/m)-1}t^{m-1}F^m(x)\,,
$$
which reduces to $U_t=bF$ when $G=0$. In any case it is nonnegative, $U_t\ge 0$.
We also have
$$
L_{\sigma}U^m=L_{\sigma}G+ b^mt^m L_{\sigma}F^m(x).
$$
 Then when $G=0$ we will have
 $$
U_t+ L_{\sigma}U^m =  bF + (bt)^mL_{\sigma}F^m + L_{\sigma}G\ge 0
$$
if $b>b_0$, $b>2k(bt)^m$, i.e. if $b^{m-1}t^mk<1/2$, which imposes a condition above on $t$.

On the other hand for $G>0$ we have
$$
U_t+ L_{\sigma}U^m \ge c_0+(bt)^mL_{\sigma}F^m\ge 0
$$
if $C_2(bt)^m\le c_0$. Both conditions are fulfilled if $0<t<t_1$.

Is this is the case the viscosity method will work in the region $Q=\{|x|\ge 1, 0<t<t_1\}$.
and this will prove that $U(x,t)\le u(x,t)$ in $Q$.

Indeed, the boundary condition at $r=1$ is
$$
U(1,t)\ge G(1)\ge 1.
$$

\noindent {\bf Remarks.} The rate of decay \eqref{rate2.tail} of the tail of such solutions at infinity is optimal as a consequence of the
construction of suitable sub-solutions with the same exponents in the $x$ and $t$ dependence, which is done in \cite{StanV2013}.

The counterexample is heavily technical. Surprisingly, the situation becomes much clearer when we let $m\to\infty$. This is studied in
\cite{VazMesa}.

%%%%%%%%%%%%%%%%%%%%%%%%%%%%%%%%%%%%%%%%%%%%%%%%%%%%%

\section{Comments, extensions and open problems}

-In a companion paper \cite{VazVol2} we will use  symmetrization results of this paper to obtain sharp a priori estimates with best constants for
some
functional embeddings involving the solutions of the linear fractional heat equation or its fast diffusion relative.

-As an extension of the above results, we could consider equations that involve a more general version of the fractional Laplacian operator, in the
same way that the standard symmetrization applies to elliptic equations with coefficients.

-The elliptic and parabolic counterexamples have been constructed for the problems posed on the whole space. They could also be constructed for
solutions defined on a bounded domain, say a ball, with zero Dirichlet boundary conditions. The argument is as follows: we consider the problems
posed
in a sequence of balls $B_R$ expanding so that $R\to\infty$ with same data of compact support. According to \cite{pqrv, pqrv2} the solutions $u_R$ converge to the solutions of the limit problem in the whole space. Now, for the limit equation there is a counterexample. We deduce that there is a counterexample before the limit. We leave the details to the reader.

-Another interesting problem would be obtaining a priori estimates for solutions of elliptic and parabolic problems of this type with Neumann boundary conditions using symmetrization techniques. A good indication is that conservation of mass is true for both elliptic and parabolic problems.

Let us now list some open problems that have arisen in the course of the work:

-We do not know how to deal with concave nonlinearities $A$ in bounded domains.

-We do not know how to do the elliptic or parabolic comparison in the case of more general function $A$, if it  is neither concave or convex.

-Finally, we wonder if there is a partial or alternative theory that replaces the failure of the concentration comparison result for the fractional
porous medium equation, i.e. the equation $$\partial_t u +(-\Delta)^{\sigma/2}u^m=0$$
with $m>1$.

\

%%%%%%%%%%%%%%%%%%%%%%%%%%%%%%%%%%%%%%%%%%%%%%%%%%%%%%%%%%%%%%%%%%%

\noindent {\large\bf Acknowledgments}

\noindent  Both authors  partially supported by the Spanish project MTM2011-24696.  We thank Felix del Teso for the  computations supporting Figure 1.

\

\medskip

%%%%%%%%%%%%%%%%%%%%%%%%%%%%%%%%%%%%%%%%%%%%%%%

{\small

%%%%%%%%%%%%%%%%%%%%

}


\begin{thebibliography}{99}


\bibitem{AT} {\sc S. Abe, S. Thurner.}  {\sl Anomalous diffusion in view of
Einstein's 1905 theory of Brownian motion}. Physica A 356 (2005),
no. 2-4, 403--407.

\bibitem{Applebaum} {\sc D. Applebaum}. {\sl \lq\lq L\'{e}vy processes and stochastic calculus''}.
Second edition. Cambridge Studies in Advanced Mathematics, 116. Cambridge
University Press, Cambridge, 2009.

\bibitem{ATLD}{\sc A. Alvino, G. Trombetti, J.~I. Diaz and P.~L. Lions}.
{\sl Elliptic equations and Steiner symmetrization},
  Comm. Pure Appl. Math. {\bf 49} (1996), no. 3, 217-236.

\bibitem{ATL90}{\sc A. Alvino, G. Trombetti, and P.~L. Lions}.
{\sl Comparison results for elliptic and parabolic equations via Schwarz symmetrization},
 Annales I. H. Poincar\'e {\bf 7}, 2(1990), 37--65.

\bibitem{AlvVolpVolz1}
{\sc A.~Alvino, R.~Volpicelli, and B.~Volzone}, {\em Sharp estimates for
  solutions of parabolic equations with a lower order term}, J. Appl. Funct.
  Anal., {\bf 3} (2008), 61--88.

\bibitem{AlvVolpVolz2}{\sc A.~Alvino, R.~Volpicelli, and B.~Volzone}, {\em Comparison results for solutions of nonlinear parabolic equations},
    Complex
    Variables and Elliptic Equations, {\bf 55} (2010), 431--443

\bibitem{Bandle}{\sc C. Bandle}.
{\sl Isoperimetric inequalities and applications.}
Monographs and Studies in Mathematics, 7. Pitman
(Advanced Publishing Program), Boston, Mass.-London, 1980.

\bibitem{Band2}{\sc C. Bandle}.
{\sl On symmetrizations in parabolic equations},
J. Analyse Math.  {\bf 30}, (1976), 98--112.

\bibitem{Barbu}{\sc V. Barbu}.
\lq\lq Nonlinear Semigroups and Differential Equations in Banach Spaces'', Noordhoff,
Leyden, 1975.


\bibitem{BeTh} {\sc Ph. B\'{e}nilan}.
{\sl Equations d'\'evolution dans un espace de Banach quelconque
et applications}, Ph. D. Thesis, Univ. Orsay, 1972 (in French).

\bibitem{Bertoin} {\sc J. Bertoin.}
\lq\lq L\'{e}vy processes''. Cambridge Tracts in Mathematics,
121. Cambridge University Press, Cambridge, 1996. ISBN: 0-521-56243-0.



\bibitem{BBC1975} {\sc P. B\'enilan, H. Brezis, M.~G. Crandall.}
{\sl A semilinear
equation in $L\sp{1}(R\sp{N})$,}  Ann. Scuola Norm. Sup. Pisa
Cl. Sci. (4) {\bf 2} (1975), 523--555.


\bibitem{BV2012}
{\sc  M. Bonforte,  J.~L. V{\'a}zquez},
{\sl Quantitative Local and Global  A Priori Estimates for Fractional Nonlinear Diffusion Equations.} In arXiv:1210.2594.



\bibitem{Colorado}
{\sc C.~Br\"andle, E.~Colorado, and A.~de~Pablo}, {\sl A concave-convex elliptic
  problem involving the fractional laplacian}, Proceedings of the Royal Society of Edinburgh  {\bf 143A}, (2013), 39--71.

\bibitem{CT} {\sc X. Cabr{\'e} and J.~G. Tan}.
{\sl Positive solutions of nonlinear problems involving the square
root of the Laplacian},  Adv. Math. {\bf 224}, 5 (2010),
{2052--2093},

\bibitem{CaffS} {\sc L.~A. Caffarelli, L. Silvestre}.
{\sl An extension problem related to the fractional Laplacian},  Comm. Partial Differential Equations  32 (2007),
no. 7-9, 1245--1260.

\bibitem{Chong}
{\sc  K.M. Chong},
{\sl Some extensions of a theorem of Hardy, Littlewood and P\'olya and their applications}, Canad. J. Math. 26 (1974), 1321--1340.


\bibitem{Cr86} {\sc M.~G. Crandall} {\sl Nonlinear Semigroups and Evolution Governed by Accretive
Operators.} In Proceedings of Symposium in Pure Math., Part I (F. Browder, ed.)
A.M.S., Providence 1986, 305--338.

\bibitem{CL71} {\sc M.~G. Crandall, T.~M. Liggett.}
{\sl Generation of semi-groups of nonlinear transformations on general
Banach spaces.}  {Amer. J. Math.} {\bf 93} (1971) 265--298.

\bibitem{Diaz1} {\sc J.~I.~Diaz}.
{\sl Symmetrization of nonlinear elliptic and parabolic problems and applications: a particular overview},
Proc. European Conf. on Elliptic and Parabolic Problems, Progress in partial differential equations: elliptic and parabolic problems
(Pont-\`{a}-Mousson, 1991),
1-16, Pitman Res. Notes Math.
Ser. {\bf 266}, Longman Sci. Tech., Harlow, 1992.

\bibitem{BV} {\sc G. Di Blasio and B. Volzone}.
{\sl Comparison and regularity results for the fractional
Laplacian via symmetrization methods},
J. Differential Equations  {\bf 253}, 9  (2012),
2593--2615.


\bibitem{pqrv} {\sc  A. de Pablo, F. Quir\'os, A. Rodr\'{\i}guez, and J.~L. V\'azquez}.
  {\sl A fractional porous medium equation.} Adv. Math. {\bf 226} (2011), no.~2, 1378--1409.

\bibitem{pqrv2} {\sc  A. de Pablo, F. Quir\'os, A. Rodr\'{\i}guez, and J.~L. V\'azquez}.
    {\sl A general fractional porous medium equation.}
    Comm. Pure Appl. Math. {\bf 65} (2012), no.~9, 1242--1284.

\bibitem{pqrv3} {\sc  A. de Pablo, F. Quir\'os, A. Rodr\'{\i}guez, and J.~L. V\'azquez}.
    {\sl Classical solutions for a logarithmic fractional diffusion equation},  Preprint.



\bibitem{pqrv4} {\sc  A. de Pablo, F. Quir\'os, A. Rodr\'{\i}guez, and J.~L. V\'azquez}.
{\sl Smooth solutions for nonlinear fractional diffusion equations}, in preparation.


\bibitem{MR1649548}
{\sc V.~Ferone and A.~Mercaldo}. {\sl A second order derivation formula for
  functions defined by integrals}, C. R. Acad. Sci. Paris S\'er. I Math. {\bf 326}
  (1998), 549--554.

\bibitem{FerMess}
{\sc V.~Ferone and B.~Messano}. {\sl Comparison and existence results for classes of nonlinear elliptic equations with general growth in the
gradient},
Adv.
Nonlinear Stud.  {\bf 7} (2007) no. 1, 31--46.

\bibitem{Getoor} {\sc R. K. Getoor.} {\sl First passage times for symmetric stable processes in space}, Trans. Amer. Math. Soc. 101 (1961),
    75--90.

\bibitem{MR0046395}
{\sc G.~H. Hardy, J.~E. Littlewood, and G.~P{\'o}lya}. {\sl Some simple inequalities satisfied by convex functions},
Messenger Math. {\bf 58 } (1929), pp. 145--152. {\sl ``Inequalities''},
  Cambridge University Press, 1952,  2d ed.


\bibitem{JKOlla} {\sc M. Jara, T. Komorowski, S. Olla.}
{\sl Limit theorems for additive functionals of a Markov chain}.
Ann. Appl. Probab. {\bf 19} (6) (2009) 2270--2300.

\bibitem{Kawohl}
{\sc B.~Kawohl}. {\sl Rearrangements and convexity of level sets in PDE},
  Lecture Notes in Mathematics, vol.~1150, Springer-Verlag, Berlin, 1985.


\bibitem{Landkof72} {\sc N.~S. Landkof. }
    \lq\lq Foundations of modern potential theory''.
  Die Grundlehren der mathematischen Wissenschaften, Band 180. Springer-Verlag, New York-Heidelberg, 1972.

\bibitem{Maz}
{\sc V.~G. Maz'ja}. {\sl Weak solutions of the Dirichlet and Neumann problems},  Trudy Moskov. Mat. Ob\u{s}u{c}. {\bf 20} (1969),   137--172.
 in  Russian

\bibitem{MMM}  {\sc A. Mellet, S. Mischler, C. Mouhot.}
{\sl Fractional diffusion limit for collisional kinetic equations}. Preprint, http://arxiv.org/abs/0809.2455.

\bibitem{Mossino}
{\sc J.~Mossino and J.-M. Rakotoson}. {\sl Isoperimetric inequalities in
  parabolic equations}, Ann. Scuola Norm. Sup. Pisa Cl. Sci. (4), 13 (1986),
  51--73.

\bibitem{PS1951}
{\sc G. P\'olya and C. Szeg\"o}. {\sl ``Isoperimetric inequalities in Mathematical Physics'',}
Annals of Mathematics Studies, vol. 27, Princeton University Press, Princeton, N.J., 1951.

 \bibitem{Reyes} { \sc   G. Reyes, J. L. V{\'a}zquez.} {\sl A weighted
symmetrization for nonlinear elliptic and parabolic equations,}
Journal European Mathematical Society {\bf 8} (2006), 531--554.

\bibitem{StanV2013} {\sc D. Stan, J.~L. V\'azquez}.
 {\sl Fisher-KPP equations with nonlinear fractional diffusion.} Preprint, 2013.

\bibitem{Stein70} {\sc E.~M. Stein}.
   \lq\lq Singular integrals and differentiability properties of functions'',
    Princeton Mathematical Series, No. 30 Princeton University Press, Princeton, N.J. 1970.

\bibitem{Talenti1} {\sc G. Talenti.} {\sl Elliptic equations and rearrangements,} Ann. Scuola Norm. Sup. (4) 3 (1976), 697--718.

\bibitem{Talenti2} {\sc G. Talenti.} {\sl Best constant in Sobolev inequality,} Ann. Mat. Pura Appl. (4) 110 (1976), 353--372.

\bibitem{Talenti3} {\sc G. Talenti.} {\sl Nonlinear elliptic equations,
rearrangements of functions and Orlicz spaces,}
Annal. Mat. Pura Appl. 4, 120 (1979), 159--184.

\bibitem{Talenti4} {\sc G. Talenti.} {\sl  Linear elliptic
  P.D.E.'s: level sets, rearrangements and a priori estimates of solutions},
  Boll. Un. Mat. Ital. B (6), 4 (1985), pp.~917--949.

\bibitem{Valdinoc} {\sc E. Valdinoci.} {\sl From the long jump random walk to the fractional Laplacian}, Bol. Soc. Esp. Mat. Apl. {\bf 49}
    (2009),
    33--44.

\bibitem{Vsym82} {\sc J. L. V\'azquez},
{\sl Sym\'etrisation pour $u_t=\Delta\varphi(u)$ et applications,}
C. R. Acad. Sc. Paris {\bf  295} (1982), pp. 71--74.

\bibitem{Vport} {\sc J.~L. V\'azquez}. {\sl Symmetrization in nonlinear parabolic equations},
Portugaliae Math. {\bf 41} (1982), pp. 339--346.

\bibitem{VANS05} {\sc J.~L. V{\'a}zquez}, {\sl Symmetrization and Mass Comparison for
Degenerate Nonlinear Parabolic and  related Elliptic Equations},
Advances in Nonlinear Studies, {\bf 5} (2005), 87--131.

\bibitem{vazquezPME} {\sc J.~L.V\'{a}zquez}.
    {\sl ``The porous medium equation. Mathematical theory'',}
    Oxford Mathematical Monographs. The Clarendon Press, Oxford University Press, Oxford, 2007.
    ISBN: 978-0-19-856903-9.

\bibitem{JLVSmoothing} {\sc J. L. V{\'a}zquez.} {\sl \lq\lq Smoothing And Decay Estimates
For Nonlinear Diffusion Equations. Equations Of Porous Medium
Type''}. Oxford Lecture Series in Mathematics and its Applications,
33. Oxford University Press, Oxford, 2006. % MR2282669 (2007K:35008)

\bibitem{Vafpme2012} {\sc J. L. V\'azquez.} {\sl Nonlinear Diffusion with Fractional Laplacian Operators},   in ``Nonlinear partial differential
    equations: the
    Abel Symposium 2010'', Holden, Helge \& Karlsen, Kenneth H. eds., Springer, 2012. Pp. 271--298.

\bibitem{vazBaren} {\sc J.~L.V\'{a}zquez}. {\sl Barenblatt  solutions and asymptotic behaviour for a  nonlinear fractional heat equation of porous
    medium
    type } Posted in arXiv:1205.6332v2.

\bibitem{VazMesa} {\sc J.~L.V\'{a}zquez}. {\sl The mesa problem for the fractional porous medium equation}, in preparation.

\bibitem{VazVol2} {\sc J.~L.~V\'azquez, B.~Volzone}. {\sl Optimal estimates for fractional  fast diffusion equations}, { in preparation.}


\bibitem{VIKH} {\sc L. Vlahos, H. Isliker, Y.  Kominis, K. Hizonidis.} {\sl Normal and anomalous Diffusion: a tutorial}. In \lq\lq Order and
    chaos'',
    10th     volume, T. Bountis (ed.), Patras University Press (2008).

\bibitem{VolpVolz} {\sc R.~Volpicelli, B.~Volzone}, {\em Comparison results for solutions of
  parabolic equations with a singular potential}, Matematiche (Catania), 62
  (2007), pp.~135--156.

 \bibitem{Wein62} {\sc H. Weinberger}. {\sl Symmetrization in uniformly elliptic problems}, Studies in
Math. Anal., Stanford Univ. Press, 1962, pp. 424--428.

\bibitem{WZ} {\sc H. Weitzner,  G.~M. Zaslavsky.}  {\sl Some applications of fractional equations.
Chaotic transport and complexity in classical and quantum
dynamics}.  Commun. Nonlinear Sci. Numer. Simul.  8  (2003),  no.
3-4, 273--281.

\end{thebibliography}
\end{document}